\documentclass[11pt,reqno]{amsart}

\title[]{Long time dynamics of the Nernst-Planck-Darcy System on $\R^3$}

\author[E. Abdo]{Elie Abdo}
\address[E. Abdo]
{
    Department of Mathematics \\
    American University of Beirut \\
    Beirut 1107-2020\\
    Lebanon.
} 
\email{ea94@aub.edu.lb}

\author[J. Germany]{Joe Germany}
\address[J. Germany]
{
    Department of Mathematics, Department of Physics \\
    American University of Beirut \\
    Beirut 1107-2020\\
    Lebanon.
}
\email{jmg15@mail.aub.edu}

\author[M.K. Hamdan]{Mohammad Khalil Hamdan}
\address[M.K. Hamdan]
{
    Department of Mathematics, Department of Philosophy \\
    American University of Beirut \\
    Beirut 1107-2020\\
    Lebanon.
}
\email{mkh28@mail.aub.edu}

\author[K. Kontar]{Kifah Kontar}
\address[K. Kontar]
{
    Department of Physics \\
    Department of Mathematics \\
    American University of Beirut \\
    Beirut 1107-2020\\
    Lebanon.
}
\email{ksa34@mail.aub.edu}

\usepackage{xfrac}

\usepackage[margin=1in]{geometry}
\usepackage{amsmath, amsthm, amssymb, yhmath, xfrac}
\usepackage{times}
\usepackage{color}
\usepackage{hyperref}
\newcommand{\pa}{\partial}
\newcommand{\la}{\label}
\newcommand{\fr}{\frac}
\newcommand{\na}{\nabla}
\newcommand{\be}{\begin{equation}}
\newcommand{\ee}{\end{equation}}
\newcommand{\ba}{\begin{array}{l}}
\newcommand{\ea}{\end{array}}

\newtheorem{prop}{Proposition}

\newcommand{\beg}{\begin}

\newcommand{\G}{\Gamma}

\newcommand{\D}{\Delta}
\newcommand{\Div}[1]{\na \cdot #1}
\newcommand{\Rint}[0]{\int_{\R^3}}
\newcommand{\RInt}[1]{\int_{\R^3} #1 \, dx}
\newcommand{\FInt}[1]{\int_{\R^3} #1 \, d\xi}
\newcommand{\Tint}[0]{\int_0^t}
\newcommand{\TInt}[1]{\int_0^t #1 \, ds}
\newcommand{\dt}[0]{\fr{d}{dt}}
\newcommand{\Lnorm}[2]{\left\lVert #1 \right\rVert_{L^{#2}}}
\newcommand{\Hnorm}[2]{\left\lVert #1 \right\rVert_{H^{#2}}}
\newcommand{\Wnorm}[3]{\left\lVert #1 \right\rVert_{W^{#2, \, #3}}}
\newcommand{\wh}[1]{\widehat{#1}}
\renewcommand{\l}{\Lambda}

\newtheorem{Thm}{Theorem}

\numberwithin{Def}{section}

\usepackage{MnSymbol,wasysym}
\usepackage{accents}

\newcommand{\R}{\mathbb R}

\def\PP{\mathbb P}

\date{\today}
\begin{document}
\begin{abstract} 
    We study ionic electrodiffusion modeled by the Nernst--Planck equations describing the evolution of \(N\) ionic species in a three-dimensional incompressible fluid flowing through a porous medium. We address the long-time dynamics of the resulting system in the three-dimensional whole space \(\mathbb{R}^3\). We prove that the \(k\)-th spatial derivatives of each ionic concentration decays to zero in \(L^2\) with a sharp rate of order \(t^{-\frac{3}{4}-\frac{k}{2}}\). Moreover, we investigate the behavior of the relative entropy associated with the model and show that it blows up in time with a sharp growth rate of order \(\log t\).
\end{abstract} 


\noindent

\maketitle
\section{Introduction}\la{intro}

We study an electrodiffusion model that describes the coupled dynamics between ionic transport, fluid motion, and self-consistent electrostatic interactions. Positively and negatively charged ions migrate under the combined effects of concentration gradients, advection by a moving fluid, and the electric field generated by the ionic charge distribution. Such coupled mechanisms lie at the heart of a broad range of phenomena in physics, chemistry, and biology, including electrokinetic flows, ion channels, and charged porous media. In particular, they have been the subject of extensive investigation in neuroscience \cite{cole1965electrodiffusion, jasielec2021electrodiffusion, koch2004biophysics, lopreore2008computational, mori2009numerical, nicholson2000diffusion, pods2013electrodiffusion, qian1989electro}, semiconductor theory \cite{biler1994debye, gajewski1986basic, mock1983analysis}, water purification, desalination, and ion separation \cite{gao2014high, lee2016membrane, lee2018diffusiophoretic, yang2019review, zhu2020ion} and selective ion membranes \cite{davidson2016dynamical, goldman1989electrodiffusion}.

Such electrodiffusion models are described by the Nernst--Planck equations governing the evolution of a system of \(N\) ionic species with concentrations \(c_1,\ldots,c_N\). Each species evolves according to 
\begin{equation}
(\partial_t + u \cdot \nabla)c_i
= D_i \,\mathrm{div}\big(\nabla c_i + z_i c_i \nabla \phi\big),
\qquad i=1,\ldots,N.
\end{equation}
Here \(c_i(x,t)\) denotes the concentration of the \(i\)-th ionic species at position \(x\) and time \(t\). The positive constants \(D_i\) represent the diffusivities, while \(z_i \in \mathbb{R}\) are the ionic valences. The transport is driven not only by diffusion and advection through the velocity field \(u\), but also by electrostatic drift induced by the electric potential \(\phi\).

The potential \(\phi\) is determined self-consistently through the Poisson equation
\begin{equation}
-\epsilon \Delta \phi = \rho :=  \sum_{i=1}^{N} z_i c_i,
\end{equation}
where \(\rho\)  is the total ionic charge density. The potential \(\phi\) is the nondimensionalized form
\begin{equation}
\phi = \frac{e}{K_{\beta} T_k}\,\Psi
\end{equation}
of the physical electric potential \(\Psi\). Here \(e\) denotes the elementary charge, \(K_{\beta}\) is Boltzmann’s constant, and \(T_k\) is the absolute temperature. The dielectric parameter
\begin{equation}
\epsilon = \frac{\mathcal{E} K_{\beta} T_k}{e^2}
= c_0 \left(\sum_{i=1}^{N} z_i^2\right)\lambda_D^2
\end{equation}
depends on the dielectric permittivity \(\mathcal{E} > 0\) of the solvent, a reference bulk concentration \(c_0 > 0\), and the Debye screening length
\begin{equation}
\lambda_D
= \sqrt{\frac{\mathcal{E} K_{\beta} T_k}{c_0 e^2 \sum_{i=1}^{N} z_i^2}}.
\end{equation}
A comprehensive discussion of the underlying physical principles of ionic electrodiffusion can be found in \cite{rubinstein1990electro}.

The ionic dynamics are coupled to the motion of the surrounding fluid. If the latter is viscous, its velocity field \(u\) satisfies the incompressible Navier--Stokes equations
\begin{equation}
\partial_t u + u \cdot \nabla u - \nu \Delta u + \nabla p
= - K_{\beta} T_k \rho \nabla \phi,
\end{equation}
together with the incompressibility constraint
\begin{equation}
\nabla \cdot u = 0.
\end{equation}
Here \(p\) denotes the pressure and \(\nu > 0\) is the kinematic viscosity. The Lorentz-type forcing term
\(- K_{\beta} T_k \rho \nabla \phi\) encodes the feedback of the electrostatic field on the fluid motion, giving rise to a fully coupled Nernst--Planck--Navier--Stokes (NPNS) system.

The Nernst–Planck–Navier–Stokes (NPNS) system has been extensively studied in the mathematical literature as a prototypical model for the nonlinear coupling between ionic electrodiffusion and incompressible fluid flow. Early works established the existence of global weak solutions and local strong solutions, emphasizing the interplay between transport, diffusion, and electrostatic forcing \cite{fischer2017global, jerome2002analytical, schmuck2009analysis}. Subsequent studies investigated regularity, uniqueness, and long-time behavior of solutions, often relying on entropy methods and energy dissipation structures inherent to the system \cite{constantin2019nernst, ryham2009existence}. The analysis was further refined to address multi-species settings, various boundary conditions, and the role of charge neutrality constraints \cite{constantin2021nernst, lee2022global}. More recent contributions focus on asymptotic stability near equilibrium, decay rates, and enhanced regularity properties, including analyticity and smoothing effects \cite{abdo2021long, abdo2022space, abdo2024long, constantin2022nernst}.

In this work, we focus on electrodiffusion in porous media, where the fluid motion is no longer governed by the Navier--Stokes equations but is instead described by Darcy's law. In this regime, the velocity field \(u\) is given by
\begin{equation}
u = - \kappa \big( \nabla p + K_{\beta} T_k \rho \nabla \phi \big),
\end{equation}
where \(\kappa > 0\) denotes the permeability of the porous medium and \(p\) is the pressure. The resulting system is referred to as the Nernst--Planck--Darcy (NPD) system. Compared to the NPNS system, the NPD model presents additional analytical challenges due to the more singular dependence of the velocity \(u\) on the electrostatic forces. This singular coupling amplifies both the nonlinearity and the nonlocal character of the Nernst--Planck equations, significantly complicating the mathematical analysis.

Early mathematical studies of Nernst–Planck–Darcy type systems focused on drift–diffusion models coupled with porous-media flow, establishing global weak solutions and basic energy dissipation mechanisms \cite{herz2016global1, herz2016global2, herz2012existence}. In the periodic setting, the case of two ionic species on the torus was addressed in \cite{ignatova2022global}, where global well-posedness and smoothness were shown to hold without smallness assumptions, relying on the special algebraic structure of the electrostatic coupling. More recently, this theory has been extended in \cite{abdo2025long} to systems with an arbitrary number of ionic species, all with equal diffusivities, on a three-dimensional torus under periodic boundary conditions. In this multi-species periodic framework, global existence, uniqueness, and asymptotic convergence to equilibrium were established, together with decay estimates. 

In this paper, we address the long-time dynamics of the NPD system for $N$ ionic species with equal diffusivities $D_1 = \dots = D_N = D$ and equal absolute valences $|z_1| = \dots = |z_N| = z$ on $\R^3$ with decay at infinity. As the physical constants $\epsilon, \kappa,  K_{\beta},$ and $T_k$ do not have any contribution to the analysis of the problem, we assume that $\epsilon = \kappa = K_{\beta} T_k = 1$ for simplicity, and we obtain the following system of partial differential equations 
\begin{align}
    &\partial_t c_i + u \cdot \na c_i - D_i \Delta c_i - z_i D_i \na \cdot (c_i \na \phi) = 0 \label{nernst_planck} \\
    &u + \nabla p = -\rho \nabla \phi \label{darcy_eq} \\
    &\na \cdot u = 0\label{div_free_eq}  \\
    &- \D \phi = \rho \label{poisson}.
\end{align}

We start by proving that the $L^2$ norm of the $k$-th order derivative of each ionic concentration decays in time with a sharp rate of decay of order
$t^{-\frac{3}{4}-\frac{k}{2}}$. The main difficulties arise from the nonlinear
and nonlocal structure of the model, generated by the advection and
electromigration terms. We show that these nonlinearities do not influence
the asymptotic decay rate of the solution; in particular, each concentration
decays at the same rate as the solution of the linear homogeneous
three-dimensional heat equation. The proof relies on the Fourier splitting
method, a bootstrapping argument, and nonlinear and nonlocal
analysis.

In addition, we investigate the entropy structure associated with the model.
As diffusion spreads the ions over increasingly large spatial regions, their
concentrations become very small, which leads to a growth of entropy. We
prove that the relative entropy of each ionic species,
\[
\mathcal{E}_i(t) = \int_{\R^3} c_i \log c_i \, dx,
\]
blows up as $t \to \infty$ with a sharp rate of growth of order $\log t$. The
argument is novel and highly nontrivial. To obtain upper bounds, we study the
evolution of the quantities ${N}_i = e^{-c\mathcal{E}_i}$ (for an appropriately chosen $c$) and prove Sobolev logarithmic inequalities on $\R^3$ to deduce that
${N}_i$ is bounded linearly in time from below. The latter yields logarithmic-in-time lower bounds for $-\mathcal{E}_i$.  On the other hand, we show that the total relative entropy
$\mathcal{E}(t)=\sum_i \mathcal{E}_i(t)$ satisfies
\[
-\mathcal{E}(t)
\le C \log\!\left(\int_{\R^3} (1+|x|)^6 c_i^2 \, dx \right),
\]
which motivates the study of polynomial moments of the solution. In this
direction, we prove that the asymptotic growth in time of $\||x|^3 c_i\|_{L^2}^2$ is of order $t^{\frac{5}{2}}$ and obtain logarithmic-in-time upper bounds for $-\mathcal{E}$.

Finally, we prove that
\[
\lim_{t \to \infty}
\left| \int_{\{|x|\le A\}} c_i \log c_i \, dx \right| = 0,
\] for any $A > 0$ and any $i \in \left\{1, \dots, N \right\}$. This shows that the contribution of the relative entropy from any fixed bounded
region vanishes as time tends to infinity. Although the total entropy grows
logarithmically in time, this growth does not originate from the core of the
domain. Instead, the ionic concentrations become locally negligible in an
entropic sense, ruling out any singular behavior on compact sets. Consequently, the entropy increase is entirely driven by the spreading of mass to large spatial scales. 

This paper is organized as follows. In section \ref{s3}, we prove that the $L^2$ norm of the $k$-th derivative of each ionic concentrations decay to $0$ with a rate of order $t^{-\frac{3}{4}-\frac{k}{2}}$. In section \ref{s4}, we prove the sharpness of that decay rate via comparison with the linear homogeneous heat equation. Finally, in section \ref{s5}, we prove that the relative entropy associated with the NPD system on $\R^3$ diverges in time with a sharp rate of order $\log t$ but converges to 0 when the spatial domain is compact.

{\bf{Notation.}} We write \(A \lesssim B\) to indicate that \(A \le C\,B\) for some positive constant \(C\) depending only on the initial data, the parameters of the problem, and universal constants. The value of \(C\) may change from line to line.

For \(f \in L^{1}(\mathbb{R}^{n})\), we denote by \(\widehat{f}\) the Fourier transform of \(f\), defined by
\[
\widehat{f}(\zeta) := \int_{\mathbb{R}^{n}} f(x)\, e^{-2\pi i\, x \cdot \zeta}\, \mathrm{d}x .
\]

Finally, we denote by \(W^{k,p}(\mathbb{R}^{3})\) the classical Sobolev spaces, with the abbreviations
\(H^{k}(\mathbb{R}^{3}) := W^{k,2}(\mathbb{R}^{3})\) and
\(L^{p}(\mathbb{R}^{3}) := W^{0,p}(\mathbb{R}^{3})\).
When no ambiguity arises, the spatial domain \(\mathbb{R}^{3}\) will be omitted from the notation.

\section{Decay in Sobolev Spaces} \la{s3}

In this section, we address the long-time behavior of solutions to the NPD system in all Sobolev spaces. As the potential $\phi$ obeys a Poisson equation, the following elliptic estimates will be used frequently in the manuscript:

\begin{prop}[Elliptic Estimates]
    Let $g$ be in $H^1(\R^3) \cap L^1(\R^3)$ such that $\int_{\R^3} g = 0$. Let $\Psi$ be a solution to the Poisson equation $- \D \Psi = g$. Then the following estimates hold:
    \begin{align}
        \Lnorm{\na \Psi}{2} &\lesssim \Lnorm{g}{1}^{\frac{2}{3}} \Lnorm{g}{2}^{\frac{1}{3}},\label{poisson1} \\
        \Lnorm{\na \Psi}{6} &\lesssim \Lnorm{g}{2}, \label{poisson2} \\
        \Lnorm{\na \Psi}{\infty} &\lesssim \Lnorm{\na \Psi}{2}^{\fr{1}{4}} \Lnorm{g}{2}^{\frac{3}{4}} + \Lnorm{g}{4}, \label{poisson3} \\
        \Lnorm{\na \Psi}{\infty} &\lesssim \Lnorm{g}{2} + \Lnorm{\na g}{2}. \label{poisson4}
    \end{align}
\end{prop}
\begin{proof}
    For \eqref{poisson1}, using the boundedness of the Riesz transform on $L^2$, the Hardy-Littlewood-Sobolev inequality, 
    followed by the log-convexity of $L^p$ norms, we have
    \begin{equation}
        \Lnorm{\na \Psi}{2} = \Lnorm{\na \l^{-2} g}{2} \lesssim \Lnorm{\l^{-1}g}{2} \lesssim \Lnorm{g}{\frac{6}{5}} \lesssim \Lnorm{g}{1}^{\frac{2}{3}} \Lnorm{g}{2}^{\frac{1}{3}}.
    \end{equation}
    For \eqref{poisson2}, using the Sobolev embedding of $H^1$ in $L^6$ and the boundedness of the Riesz transform on $L^2$ yield 
    \begin{equation}
        \Lnorm{\na \Psi}{6} \lesssim \Lnorm{\na \na \Psi}{2} = \Lnorm{\na \na \l^{-2} g}{2} \lesssim \Lnorm{g}{2}.
    \end{equation}
    For \eqref{poisson3}, by the continuous embedding of $W^{1, \, 4}$ in $L^\infty$, the three-dimensional Ladyzhenskaya interpolation inequality, and the boundedness of the Riesz transform on $L^4$, we estimate
    \begin{equation}
        \Lnorm{\na \Psi}{\infty} \lesssim \Wnorm{\na \Psi}{1}{4} \lesssim \Lnorm{\na \Psi}{4} + \Lnorm{\na \na \Psi}{4} \lesssim \Lnorm{\na \Psi}{2}^{\fr{1}{4}} \Lnorm{\D \Psi}{2}^{\fr{3}{4}} + \Lnorm{g}{4}.
    \end{equation}
    For \eqref{poisson4}, by the continuous embedding of $W^{1, \, 6}$ in $L^\infty$, the estimate \eqref{poisson2}, the Sobolev embedding theorem, and the boundedness of the Riesz transform on $L^6$, we have
    \begin{equation}
        \Lnorm{\na \Psi}{\infty} \lesssim \Wnorm{\na \Psi}{1}{6} \lesssim \Lnorm{\na \Psi}{6} + \Lnorm{\na \na \Psi}{6} \lesssim \Lnorm{g}{2} + \Lnorm{g}{6} \lesssim \Lnorm{g}{2} + \Lnorm{\na g}{2}.
    \end{equation}
\end{proof}

The NPD system has a unique local smooth solution for $H^1$ initial data, a fact that follows from a standard mollification argument and passage to the limit using the Aubin-Lions Lemma and the lower semi-continuity of the norms. Moreover, the nonnegativity of the ionic concentrations is preserved on any time interval where the existence and uniqueness of the solution are guaranteed (see \cite{constantin2019nernst}). In the following proposition, we derive a priori uniform global bounds that allow global-in-time extension and describe the behavior of the ionic concentration when time blows up. 

\begin{prop} \label{prop_3} 
Suppose $c_i(0) \ge 0$, $c_i(0) \in H^1(\R^3) \cap L^1(\R^3)$ for all $i \in \left\{1, \dots, N\right\},$ and $\int_{\R^3} \rho_0 = 0$.
Then, there exist positive constants $M_1$ and $M_2$ depending on the parameters of the problem and on the $L^1$ and $H^1$ norms of the initial concentrations $c_i(0)$, such that
\be\begin{aligned}
    \Lnorm{c_i}{2}^2 \leq \fr{M_1}{(t+1)^\fr
    {3}{2}}, \quad \Lnorm{\na c_i}{2}^2 \leq \fr{M_2}{(t+1)^\fr{5}{2}},
\end{aligned}\ee
for all $t \geq 0$ and $i \in \left\{1, \dots, N \right\}$.
\end{prop}

\begin{proof} The proof is divided into several steps. We define $\sigma = \sum\limits_{i=1}^{N} c_i$ to be the sum of the $N$ ionic concentrations. 

{\bf{Step 1. Compatibility Conditions.}}
Integrating each of the ionic concentrations over $\R^3$ and using the divergence-free condition satisfied by $u$, we obtain the conservation law 
\begin{equation}
    \RInt{c_i(t)} = \RInt{c_i(0)},
\end{equation} for each $i \in \{1, \dots, N\}$. Consequently, it follows that 
\begin{equation}
    \RInt{\rho(t)} = \RInt{\rho_0} = 0
\end{equation} after making use of the assumption $\RInt{\rho_0}  = 0$ imposed on the initial charge density. 
This guarantees compatibility conditions of the Poisson equation.

{\bf{Step 2. Uniform Bounds.}} Multiplying the equation obeyed by the ionic concentration $c_i$ by $z_i$ and summing over all indices $i \in \left\{1, \dots, N \right\}$, we infer that the charge density $\rho$ evolves in time according to
\begin{equation}
    \label{rhoPDE}
    \partial_t \rho + u \cdot \na \rho - D \D\rho = D z^2 \, \Div{(\sigma \na \phi)}.
\end{equation}
On one hand, multiplying the latter by $\l^{-2} \rho$ and integrating over $\R^3$, we obtain
\be
\begin{aligned}
    \fr{1}{2} \dfrac{d}{dt} \Lnorm{\l^{-1} \rho}{2}^2 + D \Lnorm{\rho}{2}^2 
    &= - \Rint (u \cdot \na \rho) \, \l^{-2} \rho - Dz^2 \Rint \sigma \na \phi \cdot \na \l^{-2} \rho
    \\
    &= -\Rint (u \cdot \na \rho) \l^{-2} \rho - D z^2 \Rint \sigma \na \phi \cdot \na \phi
    \\
    &= -\Rint (u \cdot \na \rho) \l^{-2} \rho  - D z^2 \Lnorm{\sqrt{\sigma} \na \phi}{2}^2 \label{rho_l-2rho}.
\end{aligned}
\ee
On the other hand, taking the $L^2$ inner product of \eqref{darcy_eq} with $u$ and integrating over $\R^3$ give
\begin{equation}
    \label{u_l2}
    \Lnorm{u}{2}^2 = -\Rint \rho \na \phi \cdot u.
\end{equation}
Combining \eqref{rho_l-2rho} and \eqref{u_l2} leads to the identity
\begin{equation}
    \Rint (u \cdot \na \rho) \l^{-2} \rho = \Rint \Div{(u \rho)} \l^{-2} \rho = - \Rint u \rho \cdot \na \l^{-2} \rho = - \Rint u \rho \cdot \na \phi = \Lnorm{u}{2}^2.
\end{equation}
Thus, we obtain
\begin{equation}
    \dfrac{1}{2} \dfrac{d}{dt} \Lnorm{\l^{-1} \rho}{2}^2 + D \Lnorm{\rho}{2}^2 + D z^2 \Lnorm{\sqrt{\sigma} \na \phi}{2}^2 + \Lnorm{u}{2}^2 = 0.
\end{equation}
Integrating in time from $0$ to $t$ and using the elliptic estimate \eqref{poisson1} yield the instantaneous bound
\be
\begin{aligned}
    &\fr{1}{2} \Lnorm{\l^{-1} \rho(t)}{2}^2 + D \TInt{\Lnorm{\rho}{2}^2} + Dz^2 \TInt {\Lnorm{\sqrt{\sigma} \na \phi}{2}^2} + \TInt{\Lnorm{u}{2}^2} 
    \\
    &\quad\quad= \fr{1}{2} \Lnorm{\l^{-1} \rho_0}{2}^2
    = \fr{1}{2} \Lnorm{\na \phi_0}{2}^2
    \lesssim \Lnorm{\rho_0}{1}^{\fr{2}{3}} \Lnorm{\rho_0}{2}^{\fr{1}{3}} \label{Uniform_Bound_1},
\end{aligned}
\ee for any $t \ge 0$. 

Next, we multiply each ionic concentration equation by $z^2$, sum over all indices $i \in \left\{1, \dots, N\right\}$, and obtain the evolution equation
\begin{equation}
    \label{sigmaPDE}
    \partial_t (z^2 \sigma) + u \cdot \na (z^2 \sigma) - D \D (z^2 \sigma) = Dz^2 \Div(\rho \na \phi) ,
\end{equation} obeyed by $\sigma$. Taking the $L^2$ inner product of the equations \eqref{rhoPDE} and  \eqref{sigmaPDE} obeyed by $\rho$ and $\sigma$ with $\rho$ and $\sigma$ respectively gives rise to the energy equalities 
\begin{equation}
    \label{rhoODE}
    \fr{1}{2} \dt \Lnorm{\rho}{2}^2 + D \Lnorm{\na \rho}{2}^2 = - D z^2 \RInt{\sigma \na \phi \cdot \na \rho} ,
\end{equation} and 
\begin{equation}
    \label{sigmaODE}
    \fr{z^2}{2} \dt \Lnorm{\sigma}{2}^2 + D z^2 \Lnorm{\na \sigma}{2}^2 = D z^2 \RInt{(\na \rho \cdot \na \phi) \sigma} + D z^2 \RInt{\rho \D \phi \sigma},
\end{equation} respectively. The nonlinear velocity terms vanish due to the divergence-free property obeyed by $u$. 
Adding \eqref{rhoODE} to \eqref{sigmaODE}, we obtain
\be
    \label{evolutionOfRhoAndSigma}
    \fr{1}{2} \dt \left( \Lnorm{\rho}{2}^2 + z^2 \Lnorm{\sigma}{2}^2 \right) + D \Lnorm{\na \rho}{2}^2 + D z^2 \Lnorm{\na \sigma}{2}^2 = - D z^2 \Rint \rho^2 \sigma 
   = -D z^2 \Lnorm{\rho \sqrt{\sigma}}{2}^2. 
\ee
Integrating in time from $0$ to $t$ gives rise to
\be
\begin{aligned}
    \Lnorm{\rho(t)}{2}^2 &+ z^2 \Lnorm{\sigma(t)}{2}^2 + 2D \TInt{\Lnorm{\na \rho(s)}{2}^2}
    \\
    &+ 2D z^2 \TInt{\Lnorm{\na \sigma(s)}{2}^2} + 2 D z^2 \TInt{\Lnorm{\rho \sqrt{\sigma}}{2}^2} = \Lnorm{\rho_0}{2}^2 + z^2 \Lnorm{\sigma_0}{2}^2 \label{Uniform_Bound_2},
\end{aligned}
\ee for any $t \ge 0$.

{\bf{Step 3. Pointwise Bounds for $|\widehat{\sigma}|$.}}
Summing the $N$ ionic concentration equations, we get
\begin{equation}
    \partial_t \sigma + u \cdot \na \sigma - D \D \sigma = D \Div (\rho \na \phi). \label{sigmaPDE_2}
\end{equation}
Taking the Fourier transform of \eqref{sigmaPDE_2}, we obtain
\begin{equation}
    \partial_t \wh{\sigma} - D \wh{\D \sigma} = D \widehat{\Div{(\rho \na \phi)}} - \wh{u \cdot \na \sigma}. \label{sigmaFourier} 
\end{equation}
Using the divergence-free property of $u$, we can deduce that
\begin{equation}
    \wh{u \cdot \na \sigma}(\xi) = \wh{\Div{(u \sigma)}} = 2 \pi i |\xi|\wh{u \sigma}.
\end{equation}
Using the latter and the boundedness of the Fourier transform of a function by its $L^1$ norm yields the differential inequality
\be
\begin{aligned}
    \partial_t \wh{\sigma} + 4 \pi^2 D |\xi|^2 \wh{\sigma} &= 2 \pi i |\xi| \wh{u \sigma} + 2 \pi i |\xi| \wh{\rho \na \phi}
    \\
    &\leq 2 \pi |\xi|(|\wh{u \sigma}| + |\wh{\rho \na \phi|})
    \\
    &\leq 2 \pi |\xi| (\Lnorm{u\sigma}{1} + \Lnorm{\rho \na \phi}{1}),
\end{aligned}
\ee
and thus we obtain
\begin{equation}
     \partial_t \wh{\sigma} + CD |\xi|^2 \wh{\sigma} \lesssim |\xi| (\Lnorm{u\sigma}{1} + \Lnorm{\rho \na \phi}{1}).
\end{equation}
We multiply both sides by the integrating factor $e^{CD |\xi|^2t}$ and integrate in time from $0$ to $t$ to get the following bound
\be\begin{aligned}
    |\wh{\sigma}(\xi, \, t)| &\lesssim |\wh{\sigma_0}(\xi)| + |\xi| \TInt{\Lnorm{u \sigma}{1}} + |\xi|\TInt{\Lnorm{\rho \na \phi}{1}}
    \\
    &\lesssim \Lnorm{\sigma_0}{1} + |\xi| \TInt{\Lnorm{u \sigma}{1}} + |\xi|\TInt{\Lnorm{\sigma \na \phi}{1}}.
\end{aligned} \ee
Using the Cauchy-Schwarz inequality, and the uniform bounds \eqref{Uniform_Bound_1} and \eqref{Uniform_Bound_2}, we have 
\begin{alignat}{2}
    |\wh{\sigma}(\xi, \, t)| &\lesssim \Lnorm{\sigma_0}{1} &+& |\xi| \left( \TInt{\Lnorm{u}{2}^2} \right)^{\frac{1}{2}} \left( \TInt{\Lnorm{\sigma}{2}^2} \right)^{\frac{1}{2}} \nonumber
    \\
    & &+& |\xi| \left( \TInt{\Lnorm{\sqrt{\sigma}}{2}^2} \right)^{\frac{1}{2}} \left( \TInt{\Lnorm{\sqrt{\sigma}\na \phi}{2}^2} \right)^{\frac{1}{2}} \label{sigma_pointwise_bound} 
    \\
    &\lesssim \Lnorm{\sigma_0}{1} &+& |\xi| \sqrt{t}, \nonumber
\end{alignat}
where $\Lnorm{\sqrt{\sigma(t)}}{2}^2 = \Lnorm{\sigma(t)}{1} = \Lnorm{\sigma_0}{1}$ is a constant.

{\bf{Step 4. $L^2$ Decay of $\rho$ and $\sigma$.}}
The $L^2$ evolution of the sum of $\rho$ and $z^2 \sigma$ is described by 
\be
\begin{aligned}
    \fr{1}{2} \dt \left( \Lnorm{\rho}{2}^2 + z^2 \Lnorm{\sigma}{2}^2 \right) + D \Lnorm{\na \rho}{2}^2 + D z^2 \Lnorm{\na \sigma}{2}^2 = -D z^2 \Lnorm{\rho \sqrt{\sigma}}{2}^2. 
\end{aligned}
\ee
In view of the non-negativity of the term $D \Lnorm{\na \rho}{2}^2$, it follows that 
\be
\begin{aligned}
    \fr{1}{2} \dt \left( \Lnorm{\rho}{2}^2 + z^2 \Lnorm{\sigma}{2}^2 \right) + D z^2 \Lnorm{\na \sigma}{2}^2 \leq -D z^2 \Lnorm{\rho \sqrt{\sigma}}{2}^2 \leq 0.
\end{aligned}
\ee
By Parseval's identity, we have
\begin{equation}
    \Lnorm{\na \sigma}{2}^2 = \Lnorm{\widehat{\na \sigma}}{2}^2 = \FInt{(\wh{\na \sigma})^2}= 4 \pi^2 \FInt{|\xi|^2 |\wh{\sigma}|^2} .
\end{equation}
Let $r: [0, \, \infty) \rightarrow [0, \, \infty)$ be a function to be determined later. 
We bound the dissipation term from below by
\begin{align}
    \Lnorm{\na \sigma}{2}^2 = C \FInt{|\xi|^2 |\wh{\sigma}|^2} \geq C \int_{|\xi| > r(t)}|\xi|^2 |\wh{\sigma}|^2 \, d\xi,
\end{align}
Thus, we obtain
\be
\begin{aligned}
    \int_{|\xi| > r(t)}|\xi|^2 |\wh{\sigma}|^2 \, d\xi &\geq r^2(t) \int_{|\xi| > r(t)} |\wh{\sigma}|^2 d \xi 
    \\
    &= r^2(t) \FInt{|\wh{\sigma}|^2} - r^2(t)\int_{|\xi| \leq r(t)} |\wh{\sigma}|^2 \, d\xi
    \\
    &= r^2(t)\Lnorm{\sigma}{2}^2 - r^2(t) \int_{|\xi| \leq r(t)} |\wh{\sigma}|^2 \, d\xi .
\end{aligned}
\ee
where we used Parseval's identity once more. Thus, we obtain the differential inequality
\begin{equation}
    \fr{1}{2} \dt \left( \Lnorm{\rho}{2}^2 + z^2 \Lnorm{\sigma}{2}^2 \right) + C D z^2 r^2(t) \Lnorm{\sigma}{2}^2 \leq C D z^2 r^2(t) \int_{|\xi| \leq r(t)} |\wh{\sigma}|^2 \, d\xi. \label{L2_fourier_splitting}
\end{equation}
By making use of the triangle inequality, we can bound
\begin{equation}
    \Lnorm{\rho}{2}^2 = \RInt{|\rho|^2} = \RInt{\left| \sum_{i=1}^N z_i c_i \right|^2} \lesssim z^2 \sum_{i=1}^N \RInt{|c_i|^2} \lesssim z^2 \Lnorm{\sigma}{2}^2 . \label{L2_rho_sigma_inequality}
\end{equation}
We split up the dissipation term in \eqref{L2_fourier_splitting} to obtain
\begin{align}
    \fr{1}{2} \dt \left( \Lnorm{\rho}{2}^2 + z^2 \Lnorm{\sigma}{2}^2 \right) + C\fr{D}{2} r^2 (t)\left(z^2 \Lnorm{\sigma}{2}^2 + z^2 \Lnorm{\sigma}{2}^2 \right) \leq CD z^2 r^2(t) \int_{|\xi| \leq r(t)} |\wh{\sigma}|^2 \, d\xi,
\end{align} 
which, by estimate \eqref{L2_rho_sigma_inequality}, simplifies to
\begin{align}
    \dt \left( \Lnorm{\rho}{2}^2 + z^2 \Lnorm{\sigma}{2}^2 \right) + c_0 D r^2 (t)\left(\Lnorm{\rho}{2}^2 + z^2 \Lnorm{\sigma}{2}^2 \right) \leq C D z^2 r^2(t) \int_{|\xi| \leq r(t)} |\wh{\sigma}|^2 \, d\xi \label{rho-sigma_fourier_inequality} .
\end{align}
Here $c_0$ is a positive constant depending only on $N$. By the pointwise bound \eqref{sigma_pointwise_bound}, Young's inequality, and Fubini's theorem for spherical coordinates, we have
\be
\begin{aligned}
    \int_{|\xi| \leq r(t)} |\wh{\sigma}(\xi, t)|^2 \, d\xi &\lesssim \int_{|\xi| \leq r(t)} \left( \Lnorm{\sigma_0}{1}^2 + |\xi|^2 t \right) \, d\xi
    \\
    &\lesssim \int_{0}^{r(t)} \int_{|\xi| = R} \left( \Lnorm{\sigma_0}{1}^2 + |\xi|^2 t \right) \, da(\xi) \, dR
    \\
    &\lesssim \int_0^{r(t)} \left( \Lnorm{\sigma_0}{1}^2 R^2 + t R^4 \right) \, dR
    \\
    &\lesssim \Lnorm{\sigma_0}{1}^2 \, r^3(t) + t \, r^5(t) .
\end{aligned}
\ee
Substituting this inequality into \eqref{rho-sigma_fourier_inequality} and choosing $r(t) = \sqrt{\fr{m}{c_0 D  (t+1)}}$ (where $m$ is a constant to be determined later) yields
\be
\begin{aligned}
    \dt \left( \Lnorm{\rho}{2}^2 + z^2 \Lnorm{\sigma}{2}^2 \right) + \fr{m}{t+1} \left(\Lnorm{\rho}{2}^2 + z^2 \Lnorm{\sigma}{2}^2 \right) &\lesssim \Lnorm{\sigma_0}{1}^2 \, r^5(t) + t \, r^7(t)
    \\
    &\lesssim \Lnorm{\sigma_0}{1}^2 \left( \fr{m}{t+1} \right)^{\fr{5}{2}} + t \left( \fr{m}{t+1} \right)^{\fr{7}{2}} .
\end{aligned}
\ee
Multiplying the differential inequality by the integrating factor 
\begin{equation}
    e^{\TInt{c_0 D r^2(s)}} = e^{\TInt{\fr{m}{s+1}}} = e^{m \ln (t+1)} = (t+1)^m,
\end{equation}
and integrating in time from $0$ to $t$, we get
\be
\begin{aligned}
    \left(\Lnorm{\rho}{2}^2 + z^2 \Lnorm{\sigma}{2}^2 \right) (t+1)^m \lesssim &\Lnorm{\rho_0}{2}^2 + z^2\Lnorm{\sigma_0}{2}^2 
    \\
    &+ \Lnorm{\sigma_0}{1}^2 \, m^{\fr{5}{2}} \TInt{(s+1)^{m-\fr{5}{2}}} + m^{\fr{7}{2}} \TInt{s (s+1)^{m-\fr{7}{2}}}.
\end{aligned}
\ee
Evaluating the integrals on the right-hand side and and dividing by $(t+1)^m$, we obtain
\be
\begin{aligned}
    \Lnorm{\rho}{2}^2 + z^2 \Lnorm{\sigma}{2}^2 &\lesssim \fr{\Lnorm{\rho_0}{2}^2 + z^2\Lnorm{\sigma_0}{2}^2}{(t + 1)^m} + \fr{\Lnorm{\sigma_0}{1}^2 \, m^\fr{5}{2} + m^\fr{7}{2}}{m - \fr{3}{2}} \fr{1}{(t+1)^\fr{3}{2}}.
\end{aligned}
\ee
Choosing $m = 2$, we obtain the following bound
\begin{equation}
    \Lnorm{\rho}{2}^2 + z^2 \Lnorm{\sigma}{2}^2 \lesssim \fr{1}{(t+1)^\fr{3}{2}}.
\end{equation} 
\noindent Finally, since $c_i \leq \sigma$, we get the desired decay of the ionic concentration as
\begin{equation}
    \Lnorm{c_i}{2}^2 \leq \Lnorm{\sigma}{2}^2 \lesssim \fr{1}{(t+1)^\fr{3}{2}},
    \label{ci_bound}
\end{equation}
proving the desired bound on $\Lnorm{c_i}{2}^2$. We now move to estimate the decay of $\na c_i$ in $L^2$.

{\bf{Step 5. $L^p$ Decay of $\sigma$.}}
We  multiply the PDE \eqref{sigmaPDE_2} by $\sigma^{p-1}$ to obtain the $L^p$ evolution of $\sigma$ as
\begin{equation}
    \fr{1}{p} \dt \Lnorm{\sigma}{p}^p + D (p-1) \Lnorm{\sigma^{\fr{p-2}{2}} \, \na \sigma }{2}^2 = D \RInt{\Div (\rho \na \phi) \, \sigma^{p-1}}.
    \label{Lp_evolution_sigma}
\end{equation}

{\bf{Substep 5.1. $L^3$ Decay of $\sigma$.}}
Specifically, the $L^3$ evolution of $\sigma$ is described by
\begin{equation}
    \fr{1}{3} \dt \Lnorm{\sigma}{3}^3 + 2D \Lnorm{\sqrt{\sigma} \, \na \sigma}{2}^2 = D \RInt{\Div(\rho \na \phi) \, \sigma^2}.
\end{equation}
Integrating by parts, then using the Holder and Young inequalities, we obtain
\vspace{0cm}
\be\begin{aligned}
    \\[-.75cm]
    \fr{1}{3} \dt \Lnorm{\sigma}{3}^3 + 2D \Lnorm{\sqrt{\sigma} \, \na \sigma}{2}^2 = -2D \RInt{\rho \na \phi \cdot \sigma \na \sigma}
    \leq D \Lnorm{\sqrt{\sigma} \, \na \sigma}{2}^2 + C \Lnorm{\rho \na \phi}{3}^2 \Lnorm{\sqrt{\sigma}}{6}^2.
    \label{L3_evolution_unsimplified}
\end{aligned}\ee
Simplifying \eqref{L3_evolution_unsimplified} and using that $\Lnorm{\sqrt{\sigma}}{6}^2 = \Lnorm{\sigma}{3}$ yield
\begin{equation}
    \dt \Lnorm{\sigma}{3}^3 + 3D \Lnorm{\sqrt{\sigma} \, \na \sigma}{2}^2 \leq C \Lnorm{\rho \na \phi}{3}^2 \Lnorm{\sigma}{3}.
    \label{L3_evolution_simplified}
\end{equation}
Now, if $\Lnorm{\sigma (\tau)}{3} = 0$ for some $\tau \ge 0$, then $\sigma(\tau) = 0$ almost everywhere in $\R^3$, and thus $\RInt{\sigma(\tau)} = 0$. By the conservation of the spatial mean of the ionic concentrations in time, we obtain
\begin{equation}
    \RInt{\sigma(\tau)} = \RInt{\sigma(t)} = 0,
\end{equation}
for all $t \geq 0$.
Thus, $\sigma(t) = 0$ almost everywhere in $\R^3$ for all $t \geq 0$. \\
Otherwise, in the case when $\Lnorm{\sigma(t)}{3} \not= 0$ for all $t \geq 0$, we divide the differential inequality \eqref{L3_evolution_simplified} describing the $L^3$ evolution of $\sigma$ by $\Lnorm{\sigma}{3}$ to obtain
\begin{equation}
    \fr{1}{\Lnorm{\sigma}{3}} \dt \Lnorm{\sigma}{3}^3 + 3D \fr{\Lnorm{\sqrt{\sigma} \, \na \sigma}{2}^2}{\Lnorm{\sigma}{3}} \leq C \Lnorm{\rho \na \phi}{3}^2,
\end{equation} 
yielding
\begin{equation}
    \dt \Lnorm{\sigma}{3}^2 \lesssim \Lnorm{\rho \na \phi}{3}^2.
\end{equation}
Using H\"older's inequality, the Sobolev inequality, the estimate \eqref{poisson2}, and the decaying bound of $c_i$ from \eqref{ci_bound}, we obtain
\begin{equation}
    \dt \Lnorm{\sigma}{3}^2 \lesssim \Lnorm{\rho \na \phi}{3}^2 \lesssim \Lnorm{\rho}{6}^2 \Lnorm{\na \phi}{6}^2 \lesssim \Lnorm{\na \rho}{2}^2 \Lnorm{\rho}{2}^2 \lesssim \fr{\Lnorm{\na \rho}{2}^2}{(t+1)^\fr{3}{2}}.
    \label{L3_sigma_DE}
\end{equation}
Now, we differentiate $(t+1) \Lnorm{\sigma}{3}^2$ in time and use the estimate \eqref{L3_sigma_DE} to get
\be
\begin{aligned}
    \dt \left((t+1) \Lnorm{\sigma}{3}^2 \right) = (t+1) \dt \Lnorm{\sigma}{3}^2 + \Lnorm{\sigma}{3}^2 \lesssim \fr{\Lnorm{\na \rho}{2}^2}{(t+1)^\fr{1}{2}} + \Lnorm{\sigma}{3}^2
    \lesssim  \Lnorm{\na \rho}{2}^2 + \Lnorm{\sigma}{3}^2.
\end{aligned}
\ee
Applying the Gagliardo-Nirenberg interpolation inequality, Young's inequality, and the decaying bound for $c_i$ in \eqref{ci_bound} gives
\be
\begin{aligned}
    \dt \left((t+1) \Lnorm{\sigma}{3}^2 \right) &\lesssim   \Lnorm{\na \rho}{2}^2 + \Lnorm{\sigma}{2} \Lnorm{\na \sigma}{2} \lesssim  \Lnorm{\na \rho}{2}^2 + \Lnorm{\sigma}{2}^2 + \Lnorm{\na \sigma}{2}^2
    \\
    &\lesssim  \Lnorm{\na \rho}{2}^2 + \fr{1}{(t+1)^\fr{3}{2}} + \Lnorm{\na \sigma}{2}^2.
\end{aligned}
\ee
Integrating in time from $0$ to $t$ and employing the bound from \eqref{Uniform_Bound_2}, we obtain
\be
\begin{aligned}
    (t+1) \, \Lnorm{\sigma (t)}{3}^2 &\lesssim \Lnorm{\sigma_0}{3}^2 +  \TInt{\Lnorm{\na \rho(s)}{2}^2} + \TInt{\Lnorm{\na \sigma(s)}{2}^2} + \TInt{\fr{1}{(s+1)^\fr{3}{2}}}
    \\
    &\lesssim \Lnorm{\sigma_0}{3}^2 +  \Lnorm{\sigma_0}{2}^2 + \Lnorm{\rho_0}{2}^2+1 \lesssim \Lnorm{\sigma_0}{2}^2 \Lnorm{\na \sigma_0}{2}^2 + \Lnorm{\sigma_0}{2}^2 + \Lnorm{\rho_0}{2}^2 \lesssim 1.
    \label{intermediate_step_for_L3_sigma}
\end{aligned}
\ee
Dividing by $(t+1)$, we establish the following decaying bound for $\Lnorm{\sigma}{3}$
\begin{equation}
    \Lnorm{\sigma(t)}{3}^2 \lesssim \fr{1}{t+1}. \label{L3_decay}
\end{equation}

{\bf{Substep 5.2. $L^4$ Decay of $\sigma$.}}
The $L^4$ evolution of $\sigma$ obeys the following differential equation
\begin{equation}
    \fr{1}{4} \dt \Lnorm{\sigma}{4}^4 + 3D \Lnorm{\sigma \, \na \sigma}{2}^2 = D \RInt{\Div{(\rho \na \phi)} \, \sigma^3} .
\end{equation}
Integrating by parts, then using the H\"older and Young inequalities, we obtain
\be\begin{aligned}
    \dt \Lnorm{\sigma}{4}^4 + 12D \Lnorm{\sigma \, \na \sigma}{2}^2 = -12D \RInt{\rho \na \phi \cdot \sigma^2 \na \sigma}
    \leq 6D \Lnorm{\sigma \, \na \sigma}{2}^2 + C \Lnorm{\rho \na \phi}{4}^2 \Lnorm{\sigma}{4}^2. \label{L4_norm_first_step}
\end{aligned}\ee
Simplifying \eqref{L4_norm_first_step} gives rise to
\be\begin{aligned}
    \dt \Lnorm{\sigma}{4}^4 \lesssim \Lnorm{\rho \na \phi}{4}^2 \Lnorm{\sigma}{4}^2 \label{L4_norm_second_step}.
\end{aligned}\ee
Following the same ideas of Substep 5.1, if $\Lnorm{\sigma(\tau)}{4} = 0$ for some $\tau \geq 0$, then $\sigma (t) = 0$ almost everywhere in $\R^3$, and so there is nothing to prove here. \\
Otherwise, we divide \eqref{L4_norm_second_step} by $\Lnorm{\sigma}{4}^2$ to reach the following differential inequality 
\be\begin{aligned}
    \dt \Lnorm{\sigma}{4}^2 \lesssim \Lnorm{\rho \na \phi}{4}^2.
\end{aligned}\ee
Using H\"older inequality, Young's inequality, the estimate \eqref{poisson3}, the Gagliardo-Nirenberg interpolation inequality, and the $L^2$ and $L^3$ decay of $c_i$ from \eqref{ci_bound} and \eqref{L3_decay} respectively, we infer that
\be\begin{aligned}
    \dt \Lnorm{\sigma}{4}^2 &\lesssim \Lnorm{\rho}{4}^2 \Lnorm{\na \phi}{\infty}^2 \lesssim \Lnorm{\rho}{4}^4 + \Lnorm{\na \phi}{\infty}^4 \lesssim \Lnorm{\rho}{4}^4 + \left( \Lnorm{\na \phi}{2}^{\fr{1}{4}} \Lnorm{\rho}{2}^{\fr{3}{4}} + \Lnorm{\rho}{4} \right)^4
    \\
    &\lesssim \Lnorm{\rho}{4}^4 + \Lnorm{\na \phi}{2} \Lnorm{\rho}{2}^3 \lesssim \Lnorm{\rho}{3}^2 \Lnorm{\na \rho}{2}^2 + \Lnorm{\na \phi}{2} \Lnorm{\rho}{2}^3
    \\
    &\lesssim  \fr{\Lnorm{\na \rho}{2}^2}{t+1} + \fr{\Lnorm{\na \phi}{2}}{(t+1)^\fr{9}{4}}. \label{L4_norm_derivative_simplified}
\end{aligned}\ee
Now, after differentiating $(t+1) \Lnorm{\sigma}{4}^2$ and using the bound \eqref{L4_norm_derivative_simplified}, we reach
\be\begin{aligned}
    \dt \left( (t+1) \Lnorm{\sigma}{4}^2 \right) &= (t+1) \dt \Lnorm{\sigma}{4}^2 + \Lnorm{\sigma}{4}^2 \lesssim  \Lnorm{\na \rho}{2}^2 +\fr{\Lnorm{\na \phi}{2}}{(t+1)^\fr{5}{4}} + \Lnorm{\sigma}{4}^2 \label{L4_norm_(t+1)derivative}.
\end{aligned}\ee
We note that, by the elliptic estimate \eqref{poisson1} and the bound derived in \eqref{Uniform_Bound_1}, we have
\begin{equation}
    \Lnorm{\na \phi}{2} \lesssim \Lnorm{\rho_0}{1}^\fr{2}{3} \Lnorm{\rho_0}{2}^\fr{1}{3} \lesssim 1.
    \label{boundedness_of_grad_phi} 
\end{equation} 
We then simplify \eqref{L4_norm_(t+1)derivative} and make use of Gagliardo-Nirenberg interpolation inequality, Young's inequality, and the $L^2$ decay of $c_i$ from \eqref{ci_bound} to deduce that
\be\begin{aligned}
    \dt \left( (t+1) \Lnorm{\sigma}{4}^2 \right) &\lesssim \Lnorm{\na \rho}{2}^2 + \fr{1}{(t+1)^\fr{5}{4}} + \Lnorm{\sigma}{2}^\fr{1}{2} \Lnorm{\na \sigma}{2}^\fr{3}{2} 
    \\
    &\lesssim \Lnorm{\na \rho}{2}^2 + \fr{1}{(t+1)^\fr{5}{4}} + \Lnorm{\sigma}{2}^2 + \Lnorm{\na \sigma}{2}^2
    \\
    &\lesssim \Lnorm{\na \rho}{2}^2 + \fr{1}{(t+1)^\fr{5}{4}} + \fr{1}{(t+1)^\fr{3}{2}} + \Lnorm{\na \sigma}{2}^2.
\end{aligned}\ee
Integrating in time from $0$ to $t$, interpolating, and using the bound \eqref{Uniform_Bound_2} yield
\be\begin{aligned}
    &(t+1) \Lnorm{\sigma (t)}{4}^2 
    \\
    &\lesssim \Lnorm{\sigma_0}{4}^2 +  \TInt{\Lnorm{\na \rho(s)}{2}^2} + \TInt{\Lnorm{\na \sigma(s)}{2}^2} + \TInt{\fr{1}{(t+1)^\fr{5}{4}}} + \TInt{\fr{1}{(t+1)^\fr{3}{2}}}
    \\
    &\lesssim \Lnorm{\sigma_0}{2}^\fr{1}{2} \Lnorm{\na \sigma_0}{2}^\fr{3}{2} +  \Lnorm{\rho_0}{2}^2 + \Lnorm{\sigma_0}{2}^2  + 1.
\end{aligned}\ee
Thus, dividing through by $(t+1)$, we obtain the desired decay on the $L^4$ norm of $\sigma$
\be\begin{aligned}
    \Lnorm{\sigma (t)}{4}^2 &\lesssim \fr{\Lnorm{\sigma_0}{2}^\fr{1}{2} \Lnorm{\na \sigma_0}{2}^\fr{3}{2} + \Lnorm{\rho_0}{2}^2 + \Lnorm{\sigma_0}{2}^2 + 1}{t+1} \lesssim \fr{1}{t+1}. \label{L4_decay}
\end{aligned}\ee

{\bf{Substep 5.3. $L^6$ Decay of $\sigma$.}}
Next, we have that the $L^6$ decay of $\sigma$ is governed by the following differential equation
\be\begin{aligned}
    \fr{1}{6} \dt \Lnorm{\sigma}{6}^6 + 5D \Lnorm{\sigma^2 \na \sigma}{2}^2 = D \RInt{\Div{(\rho \na \phi)} \, \sigma^5},
\end{aligned}\ee
which after applying  H\"older and Young inequalities simplifies to
\be\begin{aligned}
    \dt \Lnorm{\sigma}{6}^6 + 30 D \Lnorm{\sigma^2 \na \sigma}{2}^2 \leq 30 D \Lnorm{\rho \na \phi}{6} \Lnorm{\sigma^2 \na \sigma}{2} \Lnorm{\sigma^2}{3} \lesssim 15D \Lnorm{\sigma^2 \na \sigma}{2}^2 + \Lnorm{\rho \na \phi}{6}^2 \Lnorm{\sigma}{6}^4. \label{L6_decay_unsimplified}
\end{aligned}\ee
Here we used the identity $\Lnorm{\sigma^2}{3} = \Lnorm{\sigma}{6}^2$. Simplifying \eqref{L6_decay_unsimplified}, we thus deduce that
\begin{equation}
    \dt \Lnorm{\sigma}{6}^6 \lesssim \Lnorm{\rho \na \phi}{6}^2 \Lnorm{\sigma}{6}^4 \label{L6_decay_simplified} .
\end{equation}
Once more, we notice that if $\Lnorm{\sigma(\tau)}{6}^4 = 0$ for some $\tau$, then $\sigma(t) = 0$ almost everywhere in $\R^3$, and the decay of $\Lnorm{\sigma}{6}$ is trivially true. \\
Otherwise, we divide \eqref{L6_decay_simplified} by $\Lnorm{\sigma}{6}^4$ and obtain
\begin{equation}
    \dt \Lnorm{\sigma}{6}^2 \lesssim \Lnorm{\rho \na \phi}{6}^2.
\end{equation}
Making use of H\"older's inequality, the estimate \eqref{poisson3}, the boundedness of $\na \phi$ in $L^2$ derived in  \eqref{boundedness_of_grad_phi}, the Sobolev embedding theorem, and the $L^2$ and $L^4$ decay of $c_i$ from \eqref{ci_bound} and \eqref{L4_decay} respectively, we establish
\be\begin{aligned}
    \dt \Lnorm{\sigma}{6}^2 &\lesssim \Lnorm{\rho}{6}^2 \Lnorm{\na \phi}{\infty}^2 \lesssim \Lnorm{\rho}{6}^2 \left(\Lnorm{\na \phi}{2}^\fr{1}{4} \Lnorm{\rho}{2}^\fr{3}{4} + \Lnorm{\rho}{4} \right)^2 
    \\
    &\lesssim \Lnorm{\rho}{6}^2 \Lnorm{\na \phi}{2}^\fr{1}{2} \Lnorm{\rho}{2}^\fr{3}{2} + \Lnorm{\rho}{6}^2 \Lnorm{\rho}{4}^2 \lesssim  \Lnorm{\na \rho}{2}^2 \Lnorm{\rho}{2}^\fr{3}{2} + \Lnorm{\na \rho}{2}^2 \Lnorm{\rho}{4}^2 
    \\
    &\lesssim \fr{\Lnorm{\na \rho}{2}^2}{(t+1)^\fr{9}{8}} +  \fr{\Lnorm{\na \rho}{2}^2}{t+1}.
\end{aligned}\ee
By differentiating $(t+1) \Lnorm{\sigma}{6}^2$ and using the Sobolev embedding theorem, it follows that
\be\begin{aligned}
    \dt \left( (t+1) \Lnorm{\sigma}{6}^2 \right) &= (t+1) \dt \Lnorm{\sigma}{6}^2 + \Lnorm{\sigma}{6}^2 \lesssim  \fr{\Lnorm{\na \rho}{2}^2}{(t+1)^\fr{1}{8}} +  \Lnorm{\na \rho}{2}^2 + \Lnorm{\sigma}{6}^2 
    \\
    &\lesssim \Lnorm{\na \rho}{2}^2 + \Lnorm{\na \sigma}{2}^2.
\end{aligned}\ee
Integrating in time from $0$ to $t$, we obtain
\be\begin{aligned}
    (t+1) \, \Lnorm{\sigma(t)}{6}^2 \lesssim \Lnorm{\sigma_0}{6}^2 + \TInt{\Lnorm{\na \rho}{2}^2} + \Tint{\Lnorm{\na \sigma}{2}^2}.
\end{aligned}\ee
Invoking the Sobolev embedding theorem and bound \eqref{Uniform_Bound_2}, we conclude that
\be\begin{aligned}
    \Lnorm{\sigma(t)}{6}^2 \lesssim \fr{\Lnorm{\na \sigma_0}{2}^2}{t+1} + \fr{\Lnorm{\rho_0}{2}^2 + \Lnorm{\sigma_0}{2}^2}{t+1} \lesssim \fr{1}{t+1}. \label{L6_decay}
\end{aligned}\ee

{\bf{Step 6. Pointwise Bounds for $|\widehat{\na c_i}|$ for $i \in \left\{1, \dots, N\right\}$.}}
Taking the gradient of the Nernst-Planck PDE \eqref{nernst_planck} and then applying the Fourier transform leads to
\begin{equation}
    \partial_t \wh{\na c_i} + |\xi|^2 \wh{\na c_i} \lesssim |\xi|^2 \left( \lvert \wh{u  c_i} \rvert + \lvert \wh{c_i \na \phi} \rvert \right) \lesssim |\xi|^2 \left(\Lnorm{u c_i}{1} + \Lnorm{c_i \na \phi}{1} \right),
\end{equation}
giving rise to the following pointwise bound
\begin{equation}
    |\wh{\na c_i}| \lesssim \Lnorm{c_i(0)}{1} |\xi| + |\xi|^2 \sqrt{t}.
    \label{pointwise_bound_grad_c_i}
\end{equation}

{\bf{Step 7. Bounds for $\int_{0}^{t} (s+1)^{\gamma} \|\na c_i\|_{L^2}^2 \, ds$.}}
We multiply the ionic concentration equation \eqref{nernst_planck} by $c_i$ and integrate spatially over $\R^3$. Using the divergence-free property by $u$, the nonlinear term in $u$ vanishes. As for the electromigration terms, we integrate by parts and apply the Cauchy-Schwarz and Young inequalities to bound it by the sum of $\frac{D}{2} \|\na c_i\|_{L^2}^2$ and a constant multiple of $\|c_i \na \phi\|_{L^2}^2$. The latter gives rise to the differential inequality 
\begin{equation}
    \dt \Lnorm{c_i}{2}^2 + D \Lnorm{\na c_i}{2}^2 \leq C \Lnorm{c_i \na \phi}{2}^2. \label{c_i_evolution}
\end{equation}
Differentiating $(s+1)^\gamma \Lnorm{c_i}{2}^2$ and using \eqref{c_i_evolution}, we have
\be\begin{aligned}
    \dt \left( (s+1)^\gamma \Lnorm{c_i}{2}^2 \right) &+ D (s+1)^\gamma \Lnorm{\na c_i}{2}^2 \lesssim \gamma (s+1)^{\gamma -1} \Lnorm{c_i}{2}^2 + (s+1)^\gamma \Lnorm{c_i \na \phi}{2}^2
    \\
    &\lesssim \gamma (s+1)^{\gamma -1} \Lnorm{c_i}{2}^2 + (s+1)^\gamma \Lnorm{c_i}{2}^2 \Lnorm{\na \phi}{\infty}^2
    \\
    &\lesssim \gamma (s+1)^{\gamma -1} \Lnorm{c_i}{2}^2 + (s+1)^\gamma \Lnorm{c_i}{2}^2 (\Lnorm{\na \phi}{2}^\fr{1}{2} \Lnorm{\rho}{2}^\fr{3}{2} + \Lnorm{\rho}{4}^2)
    \\
    &\lesssim  \gamma  (s+1)^{\gamma - \fr{5}{2}} + (s+1)^{\gamma - \fr{21}{8}} +  (s+1)^{\gamma - \fr{5}{2}} .
\end{aligned}\ee
Next, we integrate in time from $0$ to $t$, take $\gamma \ge 2$, and conclude that
\be\begin{aligned}
   \TInt{(s+1)^\gamma \Lnorm{\na c_i}{2}^2} \lesssim 1 + (t+1)^{\gamma - \frac{3}{2}}.
    \label{integral_bound_for_grad_ci}
\end{aligned}\ee

{\bf{Step 8. $L^2$ Decay of $\na c_i$ for $i \in \left\{1, \dots, N\right\}$.}}
We start by multiplyng the Nernst-Planck PDE \eqref{nernst_planck} by $\Delta c_i$ and integrate spatially over $\R^3$ to get
\be\begin{aligned}
    \dfrac{1}{2} \dt \Lnorm{\na c_i}{2}^2 + D \Lnorm{\Delta c_i}{2}^2 \leq \RInt{| u \cdot \na c_i \Delta c_i|} + Dz^2 \RInt{|\Div{(\sigma \na \phi) \Delta c_i}|} .
    \label{inequality_for_na_c_i}
\end{aligned}\ee
We estimate
\be\begin{aligned}
    D z^2 \RInt{| \Div{(\sigma \na \phi) \Delta c_i } |} \leq Dz^2 \Lnorm{\Div{(\sigma \na \phi)}}{2} \Lnorm{\Delta c_i}{2} \le \fr{D}{6} \Lnorm{\Delta c_i}{2}^2 + C\Lnorm{\Div{(\sigma \na \phi)}}{2}^2,
\end{aligned}\ee
where
\be\begin{aligned}
    \Lnorm{\Div{(\sigma \na \phi)}}{2}^2 &\lesssim \sum_{j=1}^N \Lnorm{\Div{(c_j \na \phi)}}{2}^2 \lesssim \sum_{j=1}^N \Lnorm{\na c_j \cdot \na \phi + c_j \Delta \phi}{2}^2 
    \\
    &\lesssim \sum_{j=1}^N \Lnorm{\na c_j \cdot \na \phi}{2}^2 + \Lnorm{c_j \Delta \phi}{2}^2 \lesssim \sum_{j=1}^N \Lnorm{\na c_j}{3}^2 \Lnorm{\na \phi}{6}^2 + \Lnorm{c_j}{6}^2 \Lnorm{\Delta \phi}{3}^2 
    \\
    &\lesssim \sum_{j=1}^N \Lnorm{\na c_j}{2} \Lnorm{\Delta c_j}{2} \Lnorm{\na \na \phi}{2}^2 + \Lnorm{\na c_j}{2}^2 \Lnorm{\rho}{3}^2
    \\
    &\lesssim \sum_{j=1}^N \Lnorm{\na c_j}{2} \Lnorm{\Delta c_j}{2} \Lnorm{\rho}{2}^2 + \Lnorm{c_j}{2} \Lnorm{\Delta c_j}{2} \Lnorm{\rho}{2} \Lnorm{\na \rho}{2}
    \\
    &\le \fr{D}{6N} \sum\limits_{j=1}^{N} \Lnorm{\Delta c_j}{2}^2 + C\sum_{j=1}^{N} \Lnorm{\na c_j}{2}^2 \Lnorm{\rho}{2}^4 + C\Lnorm{c_j}{2}^2 \Lnorm{\rho}{2}^2 \Lnorm{\na \rho}{2}^2
    \\
    &\le \fr{D}{6N} \sum\limits_{j=1}^{N}  \Lnorm{\Delta c_j}{2}^2 + C \sum_{j=1}^{N} \fr{\Lnorm{\na c_j}{2}^2}{(t+1)^3}.
    \label{estimate_div_sigma_grad_phi}
\end{aligned}\ee
The above calculations follow from Minkowski, H\"older, and Young inequalities, as well as interpolation inequalities, and the bound established in \eqref{ci_bound} for the $L^2$ decay of $c_i$. \\
As a consequence, we obtain 
\be\begin{aligned}
    Dz^2 \RInt{| \Div{(\sigma \na \phi)} \Delta c_i |} \le \fr{D}{6} \Lnorm{\Delta c_i}{2}^2 + \fr{D}{6N}\sum\limits_{j=1}^{N}  \Lnorm{\Delta c_j}{2}^2 + C \sum_{j=1}^{N} \fr{\Lnorm{\na c_j}{2}^2}{(t+1)^3}.
\end{aligned}\ee
Next, using the Gagliardo-Nirenberg interpolation and Young inequalities, we bound
\be\begin{aligned}
    \RInt{| u \cdot \na c_i \Delta c_i |} &\leq \Lnorm{u}{6} \Lnorm{\na c_i}{3} \Lnorm{\Delta c_i}{2} \lesssim \Lnorm{u}{6} \Lnorm{\na c_i}{2}^\fr{1}{2} \Lnorm{\Delta c_i}{2}^\fr{3}{2} 
    \\
    &\leq \fr{D}{6} \Lnorm{\Delta c_i}{2}^2 + C \Lnorm{u}{6}^4 \Lnorm{\na c_i}{2}^2.
\end{aligned}\ee
We move to estimate $\Lnorm{u}{6}^4$ using the boundedness of the Leray projector in $L^6$, the estimate \eqref{poisson3}, the boundedness of $\na \phi$ by initial datum in \eqref{boundedness_of_grad_phi}, and the decaying bounds established in \eqref{ci_bound}, \eqref{L4_decay}, and \eqref{L6_decay} for $c_i$ in $L^2$, $L^4$, and $L^6$ respectively, reaching
\be\begin{aligned}
    \Lnorm{u}{6}^4 &= \Lnorm{\PP(\rho \na \phi)}{6}^4 \lesssim \Lnorm{\rho \na \phi}{6}^4 \lesssim \Lnorm{\rho}{6}^4 \Lnorm{\na \phi}{\infty}^4  \lesssim \Lnorm{\rho}{6}^4 \left( \Lnorm{\na \phi}{2} \Lnorm{\rho}{2}^3 + \Lnorm{\rho}{4}^4 \right) \\ 
    &\lesssim \sum_{j=1}^{N} \Lnorm{c_j}{6}^4 \left(\Lnorm{\na \phi}{2} \Lnorm{\rho}{2}^3 + \Lnorm{\rho}{4}^4 \right) \lesssim \sum_{j=1}^N \fr{1}{(t+1)^2} \left(\fr{1}{(t+1)^\fr{9}{4}} + \fr{1}{(t+1)^2} \right) \\
    &\lesssim \fr{1}{(t+1)^4}.
\end{aligned}\ee
Therefore,
\be\begin{aligned}
    \RInt{| u \cdot \na c_i \Delta c_i |} \le \fr{D}{6} \Lnorm{\Delta c_i}{2}^2 + C\fr{\Lnorm{\na c_i}{2}^2}{(t+1)^4}.
    \label{estimate_u_na_c_i_laplace_c_i}
\end{aligned}\ee
Substituting the bounds \eqref{estimate_div_sigma_grad_phi} and \eqref{estimate_u_na_c_i_laplace_c_i} into \eqref{inequality_for_na_c_i}, we now obtain the differential inequality
\be\begin{aligned}
    \dt \Lnorm{\na c_i}{2}^2 + \fr{4D}{3} \Lnorm{\Delta c_i}{2}^2 - \fr{D}{3N} \sum_{j=1}^N \Lnorm{\Delta c_j}{2}^2 \lesssim  \sum_{j=1}^N \fr{\Lnorm{\na c_j}{2}^2}{(t+1)^3} + \fr{\Lnorm{\na c_i}{2}^2}{(t+1)^4}.
\end{aligned}\ee
Summing over the ionic indices (from $i=1$ to $N$), we reach the evolution inequality for the sum of the $L^2$ norms of $\na c_i$, 
\be\begin{aligned}
    \dt \sum_{i=1}^N \Lnorm{\na c_i}{2}^2 + \fr{4D}{3} \sum_{i=1}^N \Lnorm{\Delta c_i}{2}^2 - \fr{D}{3N} \sum_{i=1}^N \sum_{j=1}^{N} \Lnorm{\Delta c_j}{2}^2 \lesssim  \sum_{i=1}^N \sum_{j=1}^N \fr{\Lnorm{\na c_j}{2}^2}{(t+1)^3} +  \sum_{i=1}^N \sum_{j=1}^N \fr{\Lnorm{\na c_i}{2}^2}{(t+1)^4}.
\end{aligned}\ee
After simplification, we infer that
\be\begin{aligned}
    \dt \sum_{i=1}^N \Lnorm{\na c_i}{2}^2 + D \sum_{i=1}^N \Lnorm{\Delta c_i}{2}^2 \lesssim \fr{1}{(t+1)^3} \sum_{i=1}^N \Lnorm{\na c_i}{2}^2 + \fr{1}{(t+1)^4} \sum_{i=1}^N \Lnorm{\na c_i}{2}^2.
    \label{l2_evolution_sum_grad_c_i}
\end{aligned}\ee
We now carry out the Fourier splitting technique on the dissipation term $\sum_{i=1}^N \Lnorm{\Delta c_i}{2}^2$ similarly to Step 4 above. We first note, that by Parseval's identity, we have
\be\begin{aligned}
    \Lnorm{\Delta c_i}{2}^2 &= \Lnorm{\wh{\Delta c_i}}{2}^2 = C\FInt{|\xi|^2 |\wh{\na c_i}|^2} \geq C \int_{|\xi| > r(t)} |\xi|^2 |\wh{\na c_i}|^2 \ d\xi 
    \\
    &\geq C r^2(t) \int_{|\xi| > r(t)} |\wh{\na c_i}|^2 \ d\xi = C r^2(t) \Lnorm{\na c_i}{2}^2 - C r^2(t) \int_{|\xi| \leq r(t)} |\wh{\na c_i}|^2 \ d\xi,
\end{aligned}\ee
for some function $r(t)$ to be determined later. Substituting this into \eqref{l2_evolution_sum_grad_c_i}, we reach the energy inequality
\be\begin{aligned}
    \dt \sum_{i=1}^N \Lnorm{\na c_i}{2}^2 &+ C D r^2(t) \sum_{i=1}^N \Lnorm{\na c_i}{2}^2 
    \\
    &\lesssim C r^2(t) \sum_{i=1}^N \int_{|\xi| \leq r(t)} |\wh{\na c_i}|^2 \ d \xi + \fr{1}{(t+1)^3} \sum_{i=1}^N \Lnorm{\na c_i}{2}^2 + \fr{1}{(t+1)^4} \sum_{i=1}^N \Lnorm{\na c_i}{2}^2.
    \label{l2_evolution_sum_grad_c_i_fourier_splitting}
\end{aligned}\ee
By the pointwise bound \eqref{pointwise_bound_grad_c_i} and Fubini's theorem for spherical coordinates, we obtain
\be\begin{aligned}
    \int_{|\xi| \leq r(t)} |\wh{\na c_i}|^2 \ d \xi &\lesssim \int_{|\xi| \leq r(t)} (\Lnorm{c_i(0)}{1}^2 |\xi|^2 + |\xi|^4 t) d\xi
    \\&= \int_0^{r(t)} \int_{|\xi| = R} (\Lnorm{c_i(0)}{1}^2 |\xi|^2 + |\xi|^4 t) \ da(\xi) \ dR
    \\
    &\lesssim \Lnorm{c_i(0)}{1}^2 \ r^5(t) + t \ r^7(t).
\end{aligned}\ee
Substituting this into \eqref{l2_evolution_sum_grad_c_i_fourier_splitting}, we get
\be\begin{aligned}
    \dt \sum_{i=1}^N \Lnorm{\na c_i}{2}^2 &+ C D r^2(t) \sum_{i=1}^N \Lnorm{\na c_i}{2}^2 
    \\
    &\lesssim \sum_{i=1}^N  \ r^7(t) + \sum_{i=1}^N t \ r^9(t) + \fr{1}{(t+1)^3} \sum_{i=1}^N \Lnorm{\na c_i}{2}^2 + \fr{1}{(t+1)^4} \sum_{i=1}^N \Lnorm{\na c_i}{2}^2.
\end{aligned}\ee
We choose $r(t) = \sqrt{\fr{m}{C D (t+1)}}$, where $m$ is a constant to be determined later. Multiplying by the integrating factor $(t+1)^m$ and integrating in time from $0$ to $t$, we find that
\be\begin{aligned}
    (t+1)^m \sum_{i=1}^{N} \Lnorm{\na c_i(t)}{2}^2 \lesssim &\sum_{i=1}^N \Lnorm{\na c_i (0)}{2}^2 + m^\fr{7}{2} \fr{(t+1)^{m-\fr{5}{2}}}{m - \fr{5}{2}} + m^\fr{9}{2} \fr{(t+1)^{m - \fr{5}{2}}}{m - \fr{5}{2}} 
    \\
    &+ \sum_{i=1}^N \TInt{(s+1)^{m-3} \Lnorm{\na c_i}{2}^2} + \sum_{i=1}^N \TInt{(s+1)^{m-4} \Lnorm{\na c_i}{2}^2},
\end{aligned}\ee
which, using the bound established in \eqref{integral_bound_for_grad_ci}, yields
\be\begin{aligned}
    (t+1)^m \sum_{i=1}^{N} &\Lnorm{\na c_i(t)}{2}^2 \lesssim 
    1 +  m^\fr{7}{2}   \fr{(t+1)^{m-\fr{5}{2}}}{m - \fr{5}{2}} + m^\fr{9}{2} \fr{(t+1)^{m - \fr{5}{2}}}{m - \fr{5}{2}} 
    +    (t+1)^{m-\fr{9}{2}}
    + (t+1)^{m-\fr{11}{2}},
\end{aligned}\ee
provided that $m -3 \ge 2$ and $m - 4 \ge 2$. We choose $m=6$ and simplify the inequality to finally deduce the decaying bound for $\na c_i$ in $L^2$ for all $i \in \{0, \dots, N\}$ as
\begin{alignat}{2}
    \sum_{i=1}^N \Lnorm{\na c_i(t)}{2}^2 
    &\lesssim \fr{1}{(t+1)^\fr{5}{2}}.
    \label{grad_ci_decay}
\end{alignat}
\end{proof}

An induction argument allows boostrapping of the time decay to higher-order derivatives of the ionic concentrations, as shown in the following proposition: 

\begin{prop} \label{prop_4}
Let $k \in \mathbb{N}$ such that $k \geq 2$. Suppose $c_i(0) \ge 0$, $c_i(0) \in H^k(\R^3) \cap L^1(\R^3)$ for all $i \in \left\{1, \dots, N\right\},$ and $\int_{\R^3} \rho_0 = 0$.
Then, there exists a positive constant $M_k$, depending on the parameters of the problem and on the $L^1$ and $H^k$ norms of the initial concentrations $c_i(0)$, such that
\be\begin{aligned}
    \Lnorm{\l^k c_i}{2}^2 \leq \fr{M_k}{(t+1)^{k + \fr{3}{2}}},
\end{aligned}\ee
for all $t \geq 0$ and $i \in \left\{1, \dots, N \right\}$. 
\end{prop}

\begin{proof} The proof is divided into several steps.

{\bf{Step 1. Pointwise Bound for $|\widehat{\Lambda^k c_i}|$ for $i \in \left\{1, \dots, N \right\}.$}}
We apply the $\l^k$ operator on the Nernst-Planck PDE \eqref{nernst_planck} and take the Fourier transform of the resulting equation, reaching
\be\begin{aligned}
    \partial_t \wh{\l^k c_i} + CD |\xi|^2 \wh{\l^k c_i} \lesssim |\xi|^{k+1} \left(\Lnorm{u}{2} \Lnorm{c_i}{2} + \Lnorm{\sqrt{\sigma} \na \phi}{2}\right).
\end{aligned}\ee
After integration by parts and using the appropriate decaying bounds, we obtain the pointwise bound
\be\begin{aligned}
    \left| \wh{\l^k c_i (t)} \right| \lesssim \Lnorm{c_i (0)}{1} |\xi|^k + |\xi|^{k+1} + |\xi|^{k+1} \sqrt{t} \lesssim \Lnorm{c_i (0)}{1} |\xi|^k + |\xi|^{k+1} \sqrt{t+1}.
    \label{pointwise_bound_for_Lambda_k_c_i}
\end{aligned}\ee

{\bf{Step 2. $L^2$ Decay of $\Lambda^k c_i$ for $i \in \left\{1, \dots, N \right\}.$}} We establish the $L^2$ decay of $\Lambda^k c_i$ by induction. Indeed, we suppose that  the following estimates 
\begin{align}
    \Lnorm{\l^{k-1} c_i}{2}^2 &\lesssim \fr{1}{(t+1)^{k + \fr{1}{2}}}
    \label{decay_of_lambda_k-1_c_i},
    \\
    \TInt{(s+1)^\gamma \Lnorm{\l^k c_i (s)}{2}^2} &\lesssim 1 + (t+1)^{\gamma - k - \fr{1}{2}},
    \label{integral_bound_for_Lambda_k}
\end{align} hold with $\gamma > k + \fr{1}{2}$, and we prove that 
\begin{align}
    \Lnorm{\l^{k} c_i}{2}^2 &\lesssim \fr{1}{(t+1)^{k + \fr{3}{2}}}
    \label{decay_of_lambda_k_c_i},
    \\
    \TInt{(s+1)^\gamma \Lnorm{\l^{k+1} c_i (s)}{2}^2} &\lesssim 1 + (t+1)^{\gamma - k - \fr{3}{2}},
    \label{integral_bound_for_Lambda_k+1}
\end{align} hold for  $\gamma > k + \fr{3}{2}$.
We begin by multiplying the Nernst-Planck PDE \eqref{nernst_planck} by $\l^{2k} c_i$ and integrating spatially to obtain
\be\begin{aligned}
    \fr{1}{2} \dt \Lnorm{\l^k c_i}{2}^2 + D \Lnorm{\l^{k+1} c_i}{2}^2 \leq  D |z_i| \left| \RInt{\Div{(c_i \na \phi)} \ \l^{2k} c_i} \right| + \left| \RInt{u \cdot \na c_i \ \l^{2k} c_i} \right|.
    \label{inequality_for_lambda_k_c_i_intermediate}
\end{aligned}\ee
We simplify the first integral term of \eqref{inequality_for_lambda_k_c_i_intermediate} through integration by parts and the H\"older and Young inequalities, and we get
\be\begin{aligned}
    D |z_i| \left| \RInt{\Div{(c_i \na \phi)} \ \l^{2k} c_i} \right| &\leq D |z_i| \left| \RInt{\l^{k-1} \Div{(c_i \na \phi)} \ \l^{k+1} c_i}\right| \\
    &\leq \fr{D}{6} \Lnorm{\l^{k+1} c_i}{2}^2 + C \Lnorm{\l^k (c_i \na \phi)}{2}^2.
\end{aligned}\ee
Using the Kato-Ponce inequality (see \cite{kato1988commutator}), the estimate \eqref{poisson4}, the $L^3$ decay of $c_i$ \eqref{L3_decay}, the Sobolev embedding theorem, the $L^2$ decay of $c_i$ and $\na c_i$ found in \eqref{ci_bound} and \eqref{grad_ci_decay} respectively, we have
\be\begin{aligned}
    C \Lnorm{\l^k (c_i \na \phi)}{2}^2 &\lesssim \Lnorm{\l^k c_i}{2}^2 \Lnorm{\na \phi}{\infty}^2 + \Lnorm{c_i}{3}^2 \Lnorm{\l^k \na \phi}{6}^2 
    \\
    &\lesssim \Lnorm{\l^k c_i}{2}^2 \left(\Lnorm{\rho}{2}^2 + \Lnorm{\na \rho}{2}^2 \right) + \fr{\Lnorm{\l^k \rho}{2}^2}{t+1} \lesssim \fr{\Lnorm{\l^k c_i}{2}^2}{(t+1)^\fr{3}{2}} + \sum_{j=1}^N \fr{\Lnorm{\l^k c_j}{2}^2}{t+1}.
    \label{Lambda_k_c_i_nabla_phi}
\end{aligned}\ee
Next, we treat the second integral term of \eqref{inequality_for_lambda_k_c_i_intermediate} by using integration by parts, the divergence-free property of $u$, and the H\"older and Young inequalities, as follows
\be\begin{aligned}
    \left| \RInt{u \cdot \na c_i \ \l^{2k} c_i}\right| \leq \left| \RInt{\l^{k-1} \Div{(u c_i)}\ \l^{k+1} c_i}\right| \leq \fr{D}{6} \Lnorm{\l^{k+1} c_i}{2}^2 + C \Lnorm{\l^k (u c_i)}{2}^2.
\end{aligned}\ee
Here, to bound $\Lnorm{\l^k (u c_i)}{2}^2$, we use the Kato-Ponce inequality, the boundedness of the Leray projector in $L^2$, the continuous embedding of $W^{1,4}$ in $L^\infty$, the Gagliardo-Nirenberg inequality, the Sobolev embedding theorem, Young's inequality, the estimate \eqref{poisson4}, the bound on $\Lnorm{\l^k (c_i \na \phi)}{2}^2$ derived in \eqref{Lambda_k_c_i_nabla_phi}, and the $L^2$ decay of $c_i$ and $\na c_i$ obtained in \eqref{ci_bound} and \eqref{grad_ci_decay} respectively, to get
\be\begin{aligned}
    C &\Lnorm{\l^k (u c_i)}{2}^2 
    \lesssim \Lnorm{\l^k u}{2}^2 \Lnorm{c_i}{\infty}^2 + \Lnorm{u}{6}^2 \Lnorm{\l^k c_i}{3}^2 
    \\
    \quad&\lesssim \Lnorm{\l^k (\rho \na \phi)}{2}^2 \Wnorm{c_i}{1}{4}^2 + \Lnorm{\rho \na \phi}{6}^2 \Lnorm{\l^k c_i}{2} \Lnorm{\l^{k+1} c_i}{2} 
    \\
    &\leq C \sum_{j=1}^N \Lnorm{\l^k (c_j \na \phi)}{2}^2 \left(\Lnorm{c_i}{4}^2 + \Lnorm{\na c_i}{4}^2 \right) + \fr{D}{6} \Lnorm{\l^{k+1} c_i}{2}^2 + C \Lnorm{\rho \na \phi}{6}^4 \Lnorm{\l^k c_i}{2}^2 
    \\
    &\leq C \sum_{j=1}^N \Lnorm{\l^k (c_j \na \phi)}{2}^2 \left(\Lnorm{c_i}{2}^\fr{1}{2} \Lnorm{\na c_i}{2}^\fr{3}{2} + \Lnorm{\na c_i}{2}^\fr{1}{2} \Lnorm{\D c_i}{2}^\fr{3}{2} \right) 
    \\&\quad\quad\quad\quad+ \fr{D}{6} \Lnorm{\l^{k+1} c_i}{2}^2 + C \Lnorm{\rho}{6}^4 \Lnorm{\na \phi}{\infty}^4 \Lnorm{\l^k c_i}{2}^2 
    \\
    &\leq \fr{C}{(t+1)^\fr{9}{4}} \sum_{j=1}^N \Lnorm{\l^k (c_j \na \phi)}{2}^2 + \fr{D}{6} \Lnorm{\l^{k+1} c_i}{2}^2 + C\Lnorm{\na \rho}{2}^4 \left(\Lnorm{\rho}{2}^4 + \Lnorm{\na \rho}{2}^4 \right) \Lnorm{\l^k c_i}{2}^2 
    \\
    &\leq \fr{C}{(t+1)^\fr{9}{4}} \sum_{j=1}^N \left(\fr{\Lnorm{\l^k c_j}{2}^2}{(t+1)^\fr{3}{2}} + \sum_{m=1}^N \fr{\Lnorm{\l^k c_m}{2}^2}{t+1} \right) + \fr{D}{6} \Lnorm{\l^{k+1} c_i}{2}^2 + C \fr{\Lnorm{\l^k c_i}{2}^2}{(t+1)^8} 
    \\
    &\leq \fr{D}{6} \Lnorm{\l^{k+1} c_i}{2}^2 + C \fr{\Lnorm{\l^k c_i}{2}^2}{(t+1)^8} + C \sum_{j=1}^N \fr{\Lnorm{\l^k c_j}{2}^2}{(t+1)^\fr{13}{4}}.
    \label{Lambda_k_u_c_i}
\end{aligned}\ee
Thus, plugging \eqref{Lambda_k_c_i_nabla_phi} and \eqref{Lambda_k_u_c_i} into \eqref{inequality_for_lambda_k_c_i_intermediate}, we deduce the following 
\be\begin{aligned}
    \dt \Lnorm{\l^k c_i}{2}^2 + D \Lnorm{\l^{k+1} c_i}{2}^2 \lesssim \fr{\Lnorm{\l^k c_i}{2}^2}{(t+1)^\fr{3}{2}} + \sum_{j=1}^N \fr{\Lnorm{\l^k c_j}{2}^2}{t+1}.
    \label{differential_inequality_for_Lambda_k_c_i}
\end{aligned}\ee
To solve this differential inequality, we apply the Fourier splitting technique on $\Lnorm{\l^{k+1} c_i}{2}^2$. Applying Parseval's identity, 
\be\begin{aligned}
    \Lnorm{\l^{k+1} c_i}{2}^2 &= \Lnorm{\wh{\l^{k+1} c_i}}{2}^2 = C \FInt{|\xi|^2 \left| \wh{\l^{k} c_i} \right|^2} \geq C \int_{|\xi| > r(t)} |\xi|^2 \left| \wh{\l^{k} c_i} \right|^2 \ d\xi \\
    &\geq C r^2(t) \int_{|\xi| > r(t)} \left| \wh{\l^{k} c_i} \right|^2 \ d\xi = C r^2(t) \Lnorm{\l^{k} c_i}{2}^2  - C r^2(t) \int_{|\xi| \leq r(t)} \left| \wh{\l^{k} c_i} \right|^2 \ d\xi,
    \label{parseval_for_Lambda_k+1_c_i}
\end{aligned}\ee
with $r(t)$ to be determined later. In view of \eqref{parseval_for_Lambda_k+1_c_i}, the pointwise bound \eqref{pointwise_bound_for_Lambda_k_c_i}, and Fubini's theorem for spherical coordinates, we find that
\be\begin{aligned}
    \int_{|\xi| \leq r(t)} \left| \wh{\l^k c_i} \right|^2 \ d\xi &\lesssim \int_{|\xi| \leq r(t)} \left( \Lnorm{c_i (0)}{1}^2 |\xi|^{2k} + |\xi|^{2k+2} (t+1) \right) \ d\xi \\
    &\lesssim \int_0^{r(t)} \int_{|\xi| = R} \left( \Lnorm{c_i (0)}{1}^2 |\xi|^{2k} + |\xi|^{2k+2} (t+1) \right) \ d \sigma (\xi) \ dR \\
    &\lesssim \Lnorm{c_i (0)}{1}^2 \ r^{2k+3} (t) + (t+1) \ r^{2k+5}(t).
    \label{fourier_splitting_for_Lambda_c_i_intermediate_step}
\end{aligned}\ee
Using \eqref{fourier_splitting_for_Lambda_c_i_intermediate_step}, the differential inequality \eqref{differential_inequality_for_Lambda_k_c_i} becomes
\be\begin{aligned}
    \dt \Lnorm{\l^k c_i}{2}^2 + CD r^2(t) \Lnorm{\l^k c_i}{2}^2 \lesssim r^{2k+5} (t) + (t+1) \ r^{2k+7} (t) + \fr{\Lnorm{\l^k c_i}{2}^2}{(t+1)^\fr{3}{2}} + \sum_{j=1}^N \fr{\Lnorm{\l^k c_j}{2}^2}{t+1}.
    \label{fourier_splitting_for_Lambda_c_i_intermediate_step_2}
\end{aligned}\ee
We choose $r(t) = \sqrt{\fr{m}{CD (t+1)}}$ (where the choice of $m$ is still pending). Then, we multiply \eqref{fourier_splitting_for_Lambda_c_i_intermediate_step_2} by $(t+1)^m$ to get
\be\begin{aligned}
    \dt &\left((t+1)^m \Lnorm{\l^k c_i}{2}^2 \right) \\
    &\lesssim m^{k+\fr{5}{2}} (t+1)^{m - k - \fr{5}{2}} + m^{k + \fr{7}{2}} (t+1)^{m - k - \fr{5}{2}} + (t+1)^{m - \fr{3}{2}} \Lnorm{\l^k c_i}{2}^2 + \sum_{j=1}^N (t+1)^{m - 1} \Lnorm{\l^k c_j}{2}^2.
\end{aligned}\ee
Integrating in time from $0$ to $t$, we obtain
\be\begin{aligned}
    (t &+1)^m \Lnorm{\l^k c_i (t)}{2}^2 
    \\
    &\lesssim \Lnorm{\l^k c_i (0)}{2}^2 + \fr{(t+1)^{m - k - \fr{3}{2}}}{m - k - \fr{3}{2}} + \TInt{(t+1)^{m - \fr{3}{2}} \Lnorm{\l^k c_i}{2}^2} + \sum_{j=1}^N \TInt{(t+1)^{m-1} \Lnorm{\l^k c_j}{2}^2},
\end{aligned}\ee
which, due to the inductive hypothesis \eqref{integral_bound_for_Lambda_k}, simplifies to
\be\begin{aligned}
    (t+1)^m \Lnorm{\l^k c_i (t)}{2}^2 \lesssim 1 + \fr{(t+1)^{m - k - \fr{3}{2}}}{m - k - \fr{3}{2}} + (t+1)^{m - k - 2} + (t+1)^{m - k - \fr{3}{2}},
\end{aligned}\ee provided that $m-\frac{3}{2} > k + \frac{1}{2}$ and $m - 1 > k + \frac{1}{2}$. 
We choose $m > k + 2$ and divide by $(t+1)^m$, establishing the desired decay for $\Lnorm{\l^k c_i}{2}^2$ as in \eqref{decay_of_lambda_k_c_i}.
We still have to show that the integral bound \eqref{integral_bound_for_Lambda_k+1} holds. To do that, we start from the previously derived differential inequality for $\Lnorm{\l^k c_i}{2}^2$ found in \eqref{differential_inequality_for_Lambda_k_c_i} and use the newly established decay on $\Lnorm{\l^k c_i}{2}^2$ in \eqref{decay_of_lambda_k_c_i}, obtaining
\be\begin{aligned}
    \dt \Lnorm{\l^k c_i}{2}^2 + D \Lnorm{\l^{k+1} c_i}{2}^2 \lesssim \fr{1}{(t+1)^{k+ 3}} + \fr{1}{(t+1)^{k + \fr{5}{2}}} \lesssim \fr{1}{(t+1)^{k+ \fr{5}{2}}}.
\end{aligned}\ee
Multiplying the latter by $(s+1)^\gamma$ and using the decay of $\Lnorm{\l^k c_i}{2}^2$ \eqref{decay_of_lambda_k_c_i}, we have
\be\begin{aligned}
    \dt \left( (s+1)^\gamma \Lnorm{\l^k c_i}{2}^2 \right) + D (s+1)^\gamma \Lnorm{\l^{k+1} c_i}{2}^2 \lesssim \gamma (s+1)^{\gamma - k - \fr{5}{2}} + (s+1)^{\gamma - k - \fr{5}{2}}.
\end{aligned}\ee
Integrating in time from $0$ to $t$ and choosing $\gamma > k + \fr{3}{2}$, we derive the desired bound for $\TInt{(s+1)^\gamma \allowbreak \Lnorm{\l^{k+1} c_i}{2}^2}$, as stated in \eqref{integral_bound_for_Lambda_k+1}.
\end{proof}

\section{Sharpness of the Decay Rate} \la{s4}

In this section, we address the sharpness of the time decay of all derivatives of the ionic concentrations. 

\begin{prop} \label{prop_5}
Suppose $c_i(0) \ge 0$, $c_i(0) \in H^k(\R^3) \cap L^1(\R^3)$ for all $i \in \left\{1, \dots,N\right\},$ and $\int_{\R^3} \rho_0 = 0$. For each $i \in \left\{1, \dots, N\right\}$, we denote by $\tilde{c}_i$ the solution to the heat equation $\pa_t \tilde{c}_i - D \Delta \tilde{c}_i = 0$ on $\R^3$ (with decay at infinity), with initial data $\tilde{c}_i(0) = c_i(0)$. There exists a positive constant $O_k$ depending on the parameters of the problem and on the $L^1$ and $H^k$ norms of the initial concentrations, such that 
\be\begin{aligned}
    \Lnorm{\l^k (c_i - \tilde{c}_i)}{2}^2 \leq \fr{O_k}{(t+1)^{k+2}}
    \label{decay_bound_for_c_i-Tilde_c_i}
\end{aligned}\ee
for all $t \geq 0$ and for all $i \in \left\{1, \dots, N\right\}$. 
\end{prop}
\begin{proof}
With the aim of proving the sharpness of the decay of $\l^k c_i$ in $L^2$, we will first subtract the linear heat PDE given by $\partial_t \tilde{c}_i - D \D \tilde{c}_i = 0$ from the Nernst-Planck PDE \eqref{nernst_planck} to get
\begin{equation}
    \partial_t (c_i - \tilde{c}_i) - D \D(c_i - \tilde{c}_i) + u \cdot \na c_i - D z_i \Div{(c_i \na \phi)} = 0.
    \label{nernst_planck_and_heat}
\end{equation}
Next, we multiply the PDE \eqref{nernst_planck_and_heat} by $\l^{2k} (c_i - \tilde{c}_i)$ and integrate spatially to obtain
\be\begin{aligned}
    \fr{1}{2} \dt \Lnorm{\l^k (c_i - \tilde{c}_i)}{2}^2 + &D \Lnorm{\l^{k+1} (c_i - \tilde{c}_i)}{2}^2 
    \\
    &\leq D | z_i | \left| \RInt{\Div{(c_i \na \phi)} \ \l^{2k} (c_i - \tilde{c}_i)} \right| + \left| \RInt{u \cdot \na c_i \ \l^{2k} (c_i - \tilde{c}_i)}\right|.
    \label{sharpness_inequality_intermediate}
\end{aligned}\ee
For the first integral term of \eqref{sharpness_inequality_intermediate}, applying integration by parts and the H\"older and Young inequalities yields
\be\begin{aligned}
    D |z_i| \left| \RInt{\Div{(c_i \na \phi)} \ \l^{2k}(c_i - \tilde{c}_i)} \right| &= D |z_i| \left| \RInt{\l^{k-1} \Div{(c_i \na \phi)} \ \l^{k+1}(c_i - \tilde{c}_i)} \right| 
    \\
    &\leq \fr{D}{4} \Lnorm{\l^{k+1} (c_i - \tilde{c}_i)}{2}^2 + C \Lnorm{\l^k (c_i \na \phi)}{2}^2,
    \label{estimate_div_c_i_grad_phi_sharpness}
\end{aligned}\ee
where
\be\begin{aligned}
    \Lnorm{\l^k (c_i \na \phi)}{2}^2 &\lesssim \Lnorm{\l^k c_i}{6}^2 \Lnorm{\na \phi}{3}^2 + \Lnorm{c_i}{\infty}^2 \Lnorm{\l^k\na \phi}{2}^2 \\
    &\lesssim \Lnorm{\l^{k+1} c_i}{2}^2 \Lnorm{\na \phi}{2} \Lnorm{\na \na \phi}{2} + \left(\Lnorm{\na c_i}{2}^2 + \Lnorm{\na \na c_i}{2}^2 \right) \Lnorm{\l^{k-1} \rho}{2}^2 \\
    &\lesssim \fr{1}{(t+1)^{k + \fr{13}{4}}} + \fr{1}{(t+1)^{k+3}} \lesssim \fr{1}{(t+1)^{k+3}},
    \label{Lambda_c_i_grad_phi_estimate_2}
\end{aligned}\ee
having used the Kato-Ponce inequality, the continuous embedding of $W^{1, 6}$ in $L^\infty$, the Sobolev embedding theorem, the Gagliardo-Nirenberg interpolation inequality, the boundedness of $\Lnorm{\na \phi}{2}$ from \eqref{boundedness_of_grad_phi}, and the decay results for $c_i$. For the second integral term of \eqref{sharpness_inequality_intermediate}, using the divergence-free property of $u$, integration by parts, and the H\"older and Young inequalities, we get
\be\begin{aligned}
    \left| \RInt{u \cdot \na c_i \ \l^{2k} (c_i - \tilde{c}_i)} \right| &\leq \RInt{\left| \l^{k-1} \Div{(u c_i)} \ \l^{k+1} (c_i - \tilde{c}_i) \right|} \\
    &\leq \Lnorm{\l^k(u c_i)}{2} \Lnorm{\l^{k+1} (c_i - \tilde{c}_i)}{2} \\
    &\leq \fr{D}{4} \Lnorm{\l^{k+1} (c_i - \tilde{c}_i)}{2}^2 + C \Lnorm{\l^k (u c_i)}{2}^2,
    \label{estimate_u_grad_c_i_sharpness}
\end{aligned}\ee
where, similarly,
\be\begin{aligned}
    \Lnorm{\l^k (u c_i)}{2}^2 &\lesssim \Lnorm{\l^k u}{2}^2 \Lnorm{c_i}{\infty}^2 + \Lnorm{u}{6}^2 \Lnorm{\l^k c_i}{3}^2 \\
    &\lesssim \Lnorm{\l^k (\rho \na \phi)}{2}^2 \left(\Lnorm{\na c_i}{2}^2 + \Lnorm{\D c_i}{2}^2 \right) + \Lnorm{\rho \na \phi}{6}^2 \Lnorm{\l^k c_i}{2} \Lnorm{\l^{k+1} c_i}{2} \\
    &\lesssim \fr{1}{(t+1)^\fr{5}{2}} \sum_{j=1}^N \Lnorm{\l^k (c_j \na \phi)}{2}^2 + \fr{1}{(t+1)^{k+2}} \Lnorm{\rho}{6}^2 \Lnorm{\na \phi}{\infty}^2 \\
    &\lesssim \fr{1}{(t+1)^{k+\fr{11}{2}}} + \fr{1}{(t+1)^{k+2}} \Lnorm{\na \rho}{2}^2 \left(\Lnorm{\rho}{2}^2 + \Lnorm{\na \rho}{2}^2 \right) \\
    &\lesssim \fr{1}{(t+1)^{k+\fr{11}{2}}} + \fr{1}{(t+1)^{k+6}} \lesssim \fr{1}{(t+1)^{k+\fr{11}{2}}},
\end{aligned}\ee
by using \eqref{Lambda_c_i_grad_phi_estimate_2} and the estimate \eqref{poisson4}. Plugging back into \eqref{sharpness_inequality_intermediate}, we obtain
\be\begin{aligned}
    \dt \Lnorm{\l^k (c_i - \tilde{c}_i)}{2}^2 + D \Lnorm{\l^{k+1} (c_i - \tilde{c}_i)}{2}^2 \leq \fr{C}{(t+1)^{k + 3}}.
    \label{differential_inequality_for_Lambda_k_c_i_c_i_hat}
\end{aligned}\ee
To apply the Fourier splitting technique, we derive an improved pointwise bound for $\left| \wh{\l^{k+1} (c_i - \tilde{c}_i)} \right|$. Applying the $\l^k$ operator on \eqref{nernst_planck_and_heat} and taking the Fourier transform, we get the differential inequality
\be\begin{aligned}
    \partial_t \widehat{\l^k (c_i - \tilde{c}_i)} + CD|\xi|^2 \widehat{\l^k(c_i - \tilde{c}_i)} \lesssim |\xi|^{k+1} \left(\Lnorm{c_i \na \phi}{1} + \Lnorm{u c_i}{1} \right).   
\end{aligned}\ee
We then multiply by the integrating factor $e^{CD |\xi|^2 t}$ and integrate in time from $0$ to $t$ to get the pointwise bound
\be\begin{aligned}
    \left| \widehat{\l^k (c_i - \tilde{c}_i)} (\xi, t) \right| &\lesssim |\xi|^{k+1} \left( \TInt{\Lnorm{c_i \na \phi}{1}} + \TInt{\Lnorm{u c_i}{1}} \right) \\
    &\lesssim |\xi|^{k+1} \left(\TInt{\Lnorm{c_i}{2} \Lnorm{\na \phi}{2}} + \TInt{\Lnorm{\rho}{2} \Lnorm{\na \phi}{\infty} \Lnorm{c_i}{2}} \right) \\
    &\lesssim |\xi|^{k+1} \left( \TInt{\fr{1}{(t+1)^\fr{3}{4}}} + \TInt{\fr{\Lnorm{\rho}{2} + \Lnorm{\na \rho}{2}}{(t+1)^\fr{3}{2}}} \right) \\
    &\lesssim |\xi|^{k+1} \left( (t+1)^\fr{1}{4} + \TInt{\fr{1}{(t+1)^\fr{9}{4}}} \right) \\
    &\lesssim |\xi|^{k+1} \, (t+1)^\fr{1}{4},
    \label{pointwise_bound_Lambda_k_c_i_c_i_hat}
\end{aligned}\ee
wherein we used H\"older's inequality, the boundedness of $\Lnorm{\na \phi}{2}$ by \eqref{boundedness_of_grad_phi}, the decay bound of $\Lnorm{c_i}{2}$ and $\Lnorm{\na c_i}{2}$ by \eqref{ci_bound} and \eqref{grad_ci_decay}, and the elliptic estimate \eqref{poisson4}.
Furthermore, by Parseval's identity, the pointwise bound \eqref{pointwise_bound_Lambda_k_c_i_c_i_hat}, and Fubini's theorem for spherical coordinates, we have
\be\begin{aligned}
    \Lnorm{\l^{k+1} (c_i - \tilde{c}_i)}{2}^2 \gtrsim r^2(t) \Lnorm{\l^k (c_i - \tilde{c}_i)}{2}^2 -  (t+1)^\fr{1}{2} \  r^{2k + 7}(t),
    \label{fourier_splitting_for_Lambda_c_i_c_i_hat_intermediate_step}
\end{aligned}\ee
with $r$ to be chosen as $r(t) = \sqrt{\fr{m}{CD(t+1)}}$. In view of \eqref{fourier_splitting_for_Lambda_c_i_c_i_hat_intermediate_step}, the differential inequality \eqref{differential_inequality_for_Lambda_k_c_i_c_i_hat} becomes
\be\begin{aligned}
    \dt \Lnorm{\l^k (c_i - \tilde{c}_i)}{2}^2 + C D r^2(t) \Lnorm{\l^k (c_i - \tilde{c}_i)}{2}^2 \lesssim (t+1)^\fr{1}{2} \ r^{2k+7}(t) + \fr{1}{(t+1)^{k + 3}}.
\end{aligned}\ee
After multiplying by the integrating factor and integrating in time from $0$ to $t$, we get the desired bound \eqref{decay_bound_for_c_i-Tilde_c_i} with choice of $m > k+2$.
\end{proof}

By making use of the sharpness of the time decay of solutions to the linear homogeneous heat equation, we obtain the following proposition:

\begin{prop} \label{prop_6} Suppose $c_i(0) \ge 0$, $c_i(0) \in H^k(\R^3) \cap L^1(\R^3)$ for all $i \in \left\{1, \dots, N\right\},$ and $\int_{\R^3} \rho_0 = 0$.
There exists a time $t_0$, depending on the parameters of the problem and on the $L^1$ and $H^k$ norms of the initial concentrations, such that for all $t \geq t_0$, we have
\begin{equation}
    \Lnorm{\l^k c_i}{2}^2 \gtrsim \fr{1}{(t+1)^{k+\fr{3}{2}}}
\end{equation} for all $i \in \left\{1, \dots, N\right\}$.
\end{prop}
\begin{proof} For each $i \in \left\{1, \dots, N\right\}$, we denote by $\tilde{c}_i$ the solution to the heat equation $\pa_t \tilde{c}_i - D \Delta \tilde{c}_i = 0$ on $\R^3$ (with decay at infinity) and with the same initial data $\tilde{c}_i(0) = c_i(0)$. 
Using the reverse triangle inequality and the decaying bound \eqref{decay_bound_for_c_i-Tilde_c_i}, we have
\be\begin{aligned}
    \Lnorm{\l^k \tilde{c}_i}{2} - \Lnorm{\l^k c_i}{2} \lesssim \Lnorm{\l^k (c_i - \tilde{c}_i)}{2} \lesssim \fr{1}{(t+1)^{\fr{k}{2} + 1}},
\end{aligned}\ee
yielding, upon rearrangement and invoking the bound \eqref{decay_of_lambda_k_c_i},
\be\begin{aligned}
    \Lnorm{\l^k c_i}{2} &\gtrsim -\fr{1}{(t+1)^{\fr{k}{2} + 1}} + \Lnorm{\l^k \tilde{c}_i}{2} \gtrsim - \fr{1}{(t+1)^{\fr{k}{2} + 1}} + \fr{1}{(t+1)^{\fr{k}{2} + \fr{3}{4}}} \\
    &\geq \fr{C_1 (t+1)^{\fr{1}{4}} - C_2}{(t+1)^{\fr{k}{2} + 1}}.
\end{aligned}\ee
We now want to find a time $t_0$ and a constant $C_3$ such that, for all $t \geq t_0$, we have $C_1 (t+1)^{\fr{1}{4}} - C_2 \geq C_3 (t+1)^\fr{1}{4}$. Choosing $C_3 = \fr{C_1}{2}$ and rearranging, we obtain that the choice of $t_0$ must be $t_0 \geq \left(\fr{2C_2}{C_1} \right)^4 - 1$. And thus, we reach the inequality,
\be\begin{aligned}
    \Lnorm{\l^k c_i}{2} \geq \fr{C_1}{2 (t+1)^{\fr{k}{2} + \fr{3}{4}}},
\end{aligned}\ee
for $t \geq t_0$, as desired.
\end{proof}

\noindent In view of Propositions \ref{prop_3}, \ref{prop_4}, \ref{prop_5}, and \ref{prop_6}, we obtain the following theorem: 

\beg{Thm} Suppose $c_i(0) \ge 0$, $c_i(0) \in H^k(\R^3) \cap L^1(\R^3)$ for all $i \in \left\{1, \dots, N\right\},$ and $\int_{\R^3} \rho_0= 0.$
There exist a time $t_0$ and positive constants $N_1, N_2$ depending on the parameters of the problem and on the $L^1(\R^3)$ and $H^k(\R^3)$ norms of the initial concentrations such that 
\begin{equation}
     \fr{N_1}{(t+1)^{k+\fr{3}{2}}} \le \Lnorm{\l^k c_i}{2}^2 \le \fr{N_2}{(t+1)^{k+\fr{3}{2}}}.
\end{equation} for all $t \ge t_0$ and all $i \in \left\{1, \dots, N\right\}$.
\end{Thm}

\section{Entropy Analysis} \la{s5}

In this section, we investigate the behavior of the relative entropy associated with the NPD model on $\R^3$.

We denote the individual and total entropy quantities by $\mathcal{E}_i (t) = \RInt{c_i(x, t) \log c_i (x, t)}$ and $\mathcal{E}(t) = \sum_{i=1}^N \mathcal{E}_i (t)$.

The following proposition provides lower bounds for $-\mathcal{E}$ that grows logarithmically in time: 

\beg{prop} \label{p} Suppose $c_i(0) \ge 0$, $c_i(0) \in H^1(\R^3) \cap L^1(\R^3)$ and $\int_{\R^3} c_i(0) \log c_i(0) \in \R$ for any $i \in \left\{1, \dots, N\right\},$ and $\int_{\R^3} \rho_0 = 0$. Then, there exist positive constants $\Gamma_1$ and $\Gamma_2$ depending on the parameters of the problem, $\mathcal{E}_i(0), \, \|c_i(0)\|_{L^1}$ and $\|c_i(0)\|_{L^2}$, and a positive constant $\Gamma_3$ depending only on the diffusivity $D$ and $\|c_i(0)\|_{L^1}$ such that
\be 
-\mathcal{E}_i \ge \Gamma_1 + \frac{3}{2} \|c_i(0)\|_{L^1} \log (\Gamma_2 + \Gamma_3t)
\ee for any $t \ge 0$ and all $i \in \left\{1, \dots, N\right\}$. As a consequence, the entropy $\mathcal{E}$ obeys 
    \begin{equation}
        \lim_{t \to \infty} \mathcal{E}(t) = - \infty.
    \end{equation} 
\end{prop}
\begin{proof}
We multiply the Nernst-Planck PDE \eqref{nernst_planck} by $\log c_i$, integrate spatially over $\R^3$, integrate by parts, and use the Poisson equation \eqref{poisson} to obtain
\be\begin{aligned}
    \dt \mathcal{E}_i (t) &= - D \RInt{\dfrac{\na c_i \cdot \na c_i}{c_i}} - z_i D \RInt{\na \phi \cdot \na c_i} \\
    &= - 4D \Lnorm{\na \sqrt{c_i}}{2}^2 - z_i D \RInt{\rho c_i}.
    \label{derivative_of_entropy}
\end{aligned}\ee
Since $\displaystyle \frac{c_i \  dx}{\int_{\R^3}c_i}$ is a probability measure and $\log x$ is a concave function, we can apply Jensen's inequality to bound the entropies from above as follows
\be\begin{aligned}
    \mathcal{E}_i (t) &= \RInt{c_i \log c_i} = \left(\RInt{c_i}\right) \RInt{\dfrac{c_i}{\RInt{c_i}} \log c_i} \\
    &\leq \Lnorm{c_i}{1} \log \left( \dfrac{\RInt{c_i^2}}{\RInt{c_i}}\right) = \Lnorm{c_i}{1} \log \left( \dfrac{\Lnorm{\sqrt{c_i}}{4}^4 }{\Lnorm{c_i}{1}}\right) .
\end{aligned}\ee
In view of the Gagliardo-Nirenberg interpolation inequality, we have
\be\begin{aligned}
    \mathcal{E}_i (t) \leq \Lnorm{c_i}{1} \log \left( \dfrac{C \Lnorm{\sqrt{c_i}}{2} \Lnorm{\na \sqrt{c_i}}{2}^3}{\Lnorm{c_i}{1}} \right) 
    = \fr{3}{2} \Lnorm{c_i}{1} \log \left(\dfrac{C \Lnorm{\na \sqrt{c_i}}{2}^2}{\Lnorm{c_i}{1}^\fr{1}{3}} \right),
\end{aligned}\ee
Thus, upon exponentiation, we get
\be\begin{aligned}
   \exp{\left({-\dfrac{2 \mathcal{E}_i}{3 \Lnorm{c_i}{1}}}\right)} \ \dfrac{\Lnorm{\na \sqrt{c_i}}{2}^2}{\Lnorm{c_i}{1}^\fr{1}{3}} \geq \dfrac{1}{C}.
    \label{intermediate_entropy_bound}
\end{aligned}\ee
We then differentiate $N_i(t) = \exp\left({\fr{-2 \mathcal{E}_i}{3 \Lnorm{c_i}{1}}}\right)$ in time and use the equality \eqref{derivative_of_entropy}, the bound \eqref{intermediate_entropy_bound}, and the decaying estimate \eqref{ci_bound} to obtain
\be\begin{aligned}
    \dt N_i(t) &= - \fr{2}{3 \Lnorm{c_i}{1}} N_i (t) \ \dt \mathcal{E}_i = \dfrac{2}{3 \Lnorm{c_i}{1}} N_i (t) \left( 4D \Lnorm{\na \sqrt{c_i}}{2}^2 + z_i D \RInt{\rho c_i}\right) \\
    &\geq \dfrac{8D}{3C \Lnorm{c_i}{1}^\fr{2}{3}} + \dfrac{2 z_i D}{3 \Lnorm{c_i}{1}} N_i (t) \RInt{\rho c_i} \geq \dfrac{8D}{3C \Lnorm{c_i}{1}^\fr{2}{3}} - \dfrac{2 |z_i| D}{3 \Lnorm{c_i}{1}} N_i (t) \Lnorm{\rho}{2} \Lnorm{c_i}{2}\\
    &\geq \dfrac{8D}{3C \Lnorm{c_i}{1}^\fr{2}{3}} - \dfrac{2 |z_i| D \ \G}{3 \Lnorm{c_i}{1} (t+1)^\fr{3}{2}} N_i (t).
    \label{entropy_differential_inequality}
\end{aligned}\ee
For notational convenience, we denote $c = \fr{2 |z_i| D \ \G}{3 \Lnorm{c_i}{1}}$ and solve the differential inequality \eqref{entropy_differential_inequality} by multiplying by the integrating factor $e^{c\TInt{(s+1)^{-1.5}}} = e^{c\left(2- 2 (t+1)^{-0.5}\right)}$, which gives
\be\begin{aligned}
    N_i (t) \geq e^{-2c} \left( N_i(0)  + \fr{8D}{3 C \Lnorm{c_i}{1}^\fr{2}{3}} t\right).
\end{aligned}\ee
Applying the logarithmic function on both sides of the latter inequality yields the desired result.
\end{proof}

The following proposition provides  upper bounds for $-\mathcal{E}$ that grows logarithmically in time: 

\beg{prop} \label{pp}  Suppose $c_i(0) \ge 0$, $c_i(0) \in H^2(\R^3) \cap L^1(\R^3)$ and $\int_{\R^3} |x|^6c_i(0)^2 \in \R$ for any $i \in \left\{1, \dots, N\right\},$ and $\int_{\R^3} \rho_0 = 0$.  Then, there exist a time $T_{\G}$ depending on the $L^1$ and $H^2$ norms of the initial concentrations and positive constants $\Gamma_4$, $\Gamma_5$, and $\Gamma_6$ depending on  $\|c_i(0)\|_{L^1}$, $\|c_i(0)\|_{H^2}$, and $\||x|^3 c_i(0)\|_{L^2}$  such that
\be\begin{aligned} \label{oo}
    - \mathcal{E}(t) \leq \Gamma_4 \log (\G_5 + \G_6 \ t^\fr{15}{16})
\end{aligned}\ee for any $t \ge t_{\G}$.
\end{prop}

\begin{proof} The proof is divided into two major steps.

{\bf{Step 1. Moment Bounds.}}
Multiplying the Nernst-Planck PDE \eqref{nernst_planck} by $|x|^6 c_i$, we have
\be\begin{aligned}
    \fr{1}{2} \dt \Lnorm{|x|^3 c_i}{2}^2 &+ D\Lnorm{|x|^3 \na c_i}{2}^2 \\
    &= - D \RInt{\na c_i \cdot \na \left(|x|^6\right) c_i} - \RInt{u \cdot \na c_i \ |x|^6 c_i} - z_i D \RInt{\Div{(c_i \na \phi) \ |x|^6 c_i}} \\
    &= \text{I} + \text{II} + \text{III}.
    \label{moment_bound_inequality}
\end{aligned}\ee
We now treat each term individually. For term I, we use H\"older's inequality with exponents $\fr{3}{2}, \ 3$ and the decaying bound \eqref{ci_bound} to obtain
\be\begin{aligned}
    \text{I} &= - D \RInt{c_i \na c_i \cdot \left(6x |x|^4\right)} \leq \fr{D}{4} \Lnorm{|x|^3 \na c_i}{2}^2 + C \Lnorm{|x|^2 c_i}{2}^2 \\
    &\leq \fr{D}{4} \Lnorm{|x|^3 \na c_i}{2}^2 + C \Lnorm{|x|^3 c_i}{2}^\fr{4}{3} \Lnorm{c_i}{2}^\fr{2}{3} \leq \fr{D}{4} \Lnorm{|x|^3 \na c_i}{2}^2 + \fr{C}{(t+1)^\fr{1}{2}} \Lnorm{|x|^3 c_i}{2}^\fr{4}{3}.
    \label{aaa}
\end{aligned}\ee
As for term II, we use the divergence-free property of $u$ and integrate by parts to get
\be\begin{aligned}
    \text{II} = \RInt{u c_i \cdot \na c_i \ |x|^6} + 6 \RInt{u c_i^2 \cdot x |x|^4} = - \text{II} + 6 \RInt{u c_i^2 \cdot x |x|^4}.
\end{aligned}\ee
By making use of continuous Sobolev embeddings, the potential estimate \eqref{poisson4}, and the decaying bounds \eqref{ci_bound}, \eqref{grad_ci_decay}, and \eqref{decay_of_lambda_k_c_i}, we estimate the velocity in $L^{\infty}$ as follows 
\be\begin{aligned}
\|u\|_{L^{\infty}}
&\lesssim \|u\|_{W^{1,6}}
\lesssim \|u\|_{L^6} + \|\na u\|_{L^6} \\&
\lesssim \|\rho\|_{L^6}\|\na \phi\|_{L^{\infty}} + \|\na \rho\|_{L^6}\|\na \phi\|_{L^{\infty}} + \|\rho\|_{L^6} \|\Delta \phi\|_{L^{\infty}} \\
&\lesssim (\|\na \rho\|_{L^2} + \|\Delta \rho\|_{L^2}) \left(\|\rho\|_{L^2} + \|\na \rho\|_{L^2} \right) + \|\na \rho\|_{L^2} (\|\na \rho\|_{L^2} + \|\Delta \rho\|_{L^2})
\lesssim \frac{1}{(t+1)^{2}},
\label{estimate_u_infty}
\end{aligned}\ee 
which yields
\be\begin{aligned}
    \text{II} &= 3 \RInt{u c_i^2 \cdot x |x|^4} \lesssim \Lnorm{|x|^3 c_i}{2} \Lnorm{|x|^2 c_i}{2} \Lnorm{u}{\infty} \lesssim \Lnorm{|x|^3 c_i}{2}^\fr{5}{3} \Lnorm{c_i}{2}^\fr{1}{3} \Lnorm{u}{\infty} 
    \\
    &\lesssim \fr{1}{(t+1)^\fr{9}{4}} \Lnorm{|x|^3 c_i}{2}^\fr{5}{3},
\end{aligned}\ee
after using a similar interpolation of $\Lnorm{|x|^2 c_i}{2}$ to the one used in \eqref{aaa}.
Invoking again the potential bound \eqref{poisson4}, we similarly estimate term III and obtain   
\be\begin{aligned}
    \text{III} &= z_i D \RInt{c_i \na \phi \cdot \left(\na c_i |x|^6 + 6x |x|^4 c_i \right)} \\
    &\lesssim \Lnorm{|x|^3 \na c_i}{2} \Lnorm{|x|^3 c_i}{2} \Lnorm{\na \phi}{\infty} + \Lnorm{|x|^3 c_i}{2} \Lnorm{|x|^2 c_i}{2} \Lnorm{\na \phi}{\infty} \\
    &\leq \fr{D}{4} \Lnorm{|x|^3 \na c_i}{2}^2 + C \Lnorm{|x|^3 c_i}{2}^2 \Lnorm{\na \phi}{\infty}^2 + C \Lnorm{|x|^3 c_i}{2}^\fr{5}{3} \Lnorm{c_i}{2}^\fr{1}{3} \Lnorm{\na \phi}{\infty} \\
    &\leq \fr{D}{4} \Lnorm{|x|^3 \na c_i}{2}^2 + \fr{C}{(t+1)^\fr{3}{2}} \Lnorm{|x|^3 c_i}{2}^2 + \fr{C}{(t+1)} \Lnorm{|x|^3 c_i}{2}^\fr{5}{3}.
\end{aligned}\ee
Consequently, combining these individual estimates and using Young's inequality give rise to
\be\begin{aligned}
    \dt \Lnorm{|x|^3 c_i}{2}^2 &\lesssim \fr{1}{(t+1)^\fr{1}{2}} \Lnorm{|x|^3 c_i}{2}^\fr{4}{3} + \fr{1}{(t+1)} \Lnorm{|x|^3 c_i}{2}^\fr{5}{3} + \fr{1}{(t+1)^\fr{3}{2}} \Lnorm{|x|^3 c_i}{2}^2 \\
    &\leq \fr{C_0}{(t+1)^\fr{3}{2}} \Lnorm{|x|^3 c_i}{2}^2 + C (t+1)^\fr{3}{2},
\end{aligned}\ee
which, upon applying Gronwall's inequality, yields
\be\begin{aligned} \label{ooo}
    \Lnorm{|x|^3 c_i}{2}^2 \leq C_1 + C_2 (t+1)^\fr{5}{2},
\end{aligned}\ee
where $C_1$ and $C_2$ depend on $\Lnorm{|x|^3 c_i(0)}{2}$, $\Hnorm{c_i(0)}{2}$, $\Lnorm{c_i(0)}{1}$, and the parameters of the problem.

{\bf{Step 2. Absolute Entropy Upper Bound.}}
By the continuous embedding of $W^{1,6}$ in $L^\infty$ and the Sobolev embedding theorem, we know that
\be\begin{aligned}
    \Lnorm{c_i}{\infty} \leq \Wnorm{c_i}{1}{6} \lesssim \Lnorm{\na c_i}{2} + \Lnorm{\D c_i}{2} \lesssim \dfrac{1}{(t+1)^\fr{5}{4}} \leq \fr{1}{2},
\end{aligned}\ee
for $t \geq T_{\G}$, where $T_\Gamma$ is a constant that depends on the $L^1$ and $H^2$ norms of the initial ionic concentrations. Hence, it holds that $|c_i(x, t)|  \leq \fr{1}{2}$ for $t \geq T_\G$ and almost every $x \in \R^3$. Due to the concavity of the logarithmic function and the fact that $\frac{c_i \, dx}{\int_{\R^3} c_i}$ is a probability measure, we can make use of Jensen's inequality to bound the entropy by
\be\begin{aligned}
    \left|\mathcal{E}(t)\right| &\leq \sum_{i=1}^N \RInt{|c_i \log c_i| } = \sum_{i=1}^N \RInt{c_i \log \fr{1}{c_i}} = 4 \sum_{i=1}^N \left(\RInt{c_i} \right) \RInt{\fr{c_i}{\RInt{c_i}} \log \left( \fr{1}{c_i} \right)^\fr{1}{4}} \\
    &\lesssim \sum_{i=1}^N \Lnorm{c_i}{1} \log \left( \RInt{\fr{c_i^\fr{3}{4}}{\Lnorm{c_i}{1}}} \right),
    \label{absolute_entropy_inequality}
\end{aligned}\ee for any $t \ge t_{\G}$. 
Using H\"older's inequality with exponent $\fr{8}{3}, \, \fr{8}{5},$ and the moment bound \eqref{ooo}, we have
\be\begin{aligned}
    \RInt{c_i^\fr{3}{4}} &= \RInt{\fr{(1+|x|)^\fr{9}{4} \, c_i^\fr{3}{4}}{(1+ |x|)^\fr{9}{4}}} \leq \left( \RInt{(1+|x|)^6 \ c_i^2} \right)^\fr{3}{8} \left( \RInt{\fr{1}{(1+|x|)^\fr{18}{5}}}\right)^\fr{5}{8} \\
    &\lesssim \Lnorm{c_i}{2}^\fr{3}{4} + \Lnorm{|x|^3 c_i}{2}^\fr{3}{4} \lesssim 1 + (t+1)^\fr{15}{16}.
\end{aligned}\ee
Thus, substituting back into \eqref{absolute_entropy_inequality}, we obtain the desired estimate \eqref{oo}.
\end{proof}

When the spatial domain is bounded, the entropy decays in time to 0:

\beg{prop} \label{ppp} Let $A > 0$. Suppose $c_i(0) \ge 0$, $c_i(0) \in H^2(\R^3) \cap L^1(\R^3)$ and $\int_{\R^3} \rho_0 = 0$ for all $i \in \left\{1, \dots, N\right\}$. Then, it holds that 
    \begin{equation}
        \lim_{t \to \infty} \left| \int_{\{|x| \leq A\}} c_i \log c_i \ dx \right| = 0
    \end{equation} for all $i \in \left\{1, \dots, N\right\}$.
\end{prop}

\beg{proof}
To prove the convergence of entropy on finite spatial domains, we decompose the entropy integral over the finite sphere of radius $A$ into the following two integrals
\be\begin{aligned}
    \left| \int_{\{|x| \leq A\}} c_i \log c_i \ dx \right| &\leq \int_{\{|x| \leq A\}} \left| c_i \log c_i \right| \ dx \\
    &= \int_{\{|x| \leq A\} \cap \{c_i \leq 1\}} \left| c_i \log c_i \right| \ dx + \int_{\{|x| \leq A\} \cap \{c_i > 1\}} \left| c_i \log c_i \right| \ dx \\
    &\leq \int_{\{|x| \leq A\} \cap \{c_i \leq 1\}} c_i \log \fr{1}{c_i} \ dx + \RInt{c_i^2} \lesssim \int_{\{|x| \leq A\}} c_i^\fr{1}{2} \ dx + \Lnorm{c_i}{2}^2 \\
    &\lesssim \left( \int_{\{|x| \leq A\}} c_i^2 \ dx \right)^{\frac{1}{4}} \left(\int_{\{|x| < A\}} 1 \ dx \right)^\fr{3}{4} + \Lnorm{c_i}{2}^2 \\
    &\lesssim \left(\fr{4}{3}\pi A^3\right)^\fr{3}{4} \Lnorm{c_i}{2}^\fr{1}{2} +  \Lnorm{c_i}{2}^2 \lesssim \fr{1}{(t+1)^\fr{3}{8}} + \fr{1}{(t+1)^\fr{3}{2}} \lesssim \fr{1}{(t+1)^\fr{3}{8}}.
\end{aligned}\ee
Letting $t \to \infty$, we obtain the desired result.
\end{proof}
\noindent In view of Propositions \ref{p}, \ref{pp}, and \ref{ppp}, we obtain the following theorem: 

\begin{Thm}
Suppose $c_i(0) \ge 0$, $c_i(0) \in H^2(\R^3) \cap L^1(\R^3)$, $\int_{\R^3} c_i(0) \log c_i(0) < \infty$, $\int_{\R^3} |x|^6c_i(0)^2 \in \R$ for any $i \in \left\{1, \dots, N\right\},$ and $\int_{\R^3} \rho_0 = 0$.  Then, there exist a time $T_{\G}$ depending on the $L^1$ and $H^2$ norms of the initial concentrations and positive constants $\gamma_1, \gamma_2 >0$ depending on the initial data and the parameters of the problem such that 
\be\begin{aligned}
   -\gamma_1 (1 + \log (1+t)) \le  \mathcal{E}(t) \le -\gamma_2 (1 + \log (1+t)) 
\end{aligned}\ee for any $t \ge t_{\G}$. Moreover, it holds that 
\begin{equation}
        \lim_{t \to \infty} \left| \int_{\{|x| \leq A\}} c_i \log c_i \ dx \right| = 0
    \end{equation} for any $A > 0$ and any $i \in \left\{1, \dots, N\right\}$.
\end{Thm}

\vspace{0.5cm}

{\bf{Data Availability Statement.}} The research does not have any associated data.

\vspace{0.5cm}

{\bf{Conflict of Interest.}} The authors declare that they have no conflict of interest.

\bibliographystyle{plain}
\bibliography{bibliography}

\begin{thebibliography}{10}

\bibitem{abdo2021long}
Elie Abdo and Mihaela Ignatova.
\newblock Long time finite dimensionality in charged fluids.
\newblock {\em Nonlinearity}, 34(9):6173, 2021.

\bibitem{abdo2022space}
Elie Abdo and Mihaela Ignatova.
\newblock On the space analyticity of the nernst--planck--navier--stokes system.
\newblock {\em Journal of Mathematical Fluid Mechanics}, 24(2):51, 2022.

\bibitem{abdo2024long}
Elie Abdo and Mihaela Ignatova.
\newblock Long time dynamics of nernst-planck-navier-stokes systems.
\newblock {\em Journal of Differential Equations}, 379:794--828, 2024.

\bibitem{abdo2025long}
Elie Abdo and Đorđe Nikolić.
\newblock Long time dynamics of the three-dimensional nernst--planck--darcy model.
\newblock {\em Nonlinearity}, 38(9):095009, 2025.

\bibitem{biler1994debye}
Piotr Biler, Waldemar Hebisch, and Tadeusz Nadzieja.
\newblock The debye system: existence and large time behavior of solutions.
\newblock {\em Nonlinear Analysis: Theory, Methods \& Applications}, 23(9):1189--1209, 1994.

\bibitem{cole1965electrodiffusion}
Kenneth~S Cole.
\newblock Electrodiffusion models for the membrane of squid giant axon.
\newblock {\em Physiological Reviews}, 45(2):340--379, 1965.

\bibitem{constantin2019nernst}
Peter Constantin and Mihaela Ignatova.
\newblock On the nernst--planck--navier--stokes system.
\newblock {\em Archive for Rational Mechanics and Analysis}, 232(3):1379--1428, 2019.

\bibitem{constantin2021nernst}
Peter Constantin, Mihaela Ignatova, and Fizay-Noah Lee.
\newblock Nernst--planck--navier--stokes systems far from equilibrium.
\newblock {\em Archive for Rational Mechanics and Analysis}, 240(2):1147--1168, 2021.

\bibitem{constantin2022nernst}
Peter Constantin, Mihaela Ignatova, and Fizay-Noah Lee.
\newblock Nernst-planck-navier-stokes systems near equilibrium.
\newblock {\em Pure Appl. Funct. Anal.}, 7(1):175--196, 2022.

\bibitem{davidson2016dynamical}
Scott~M Davidson, Matthias Wessling, and Ali Mani.
\newblock On the dynamical regimes of pattern-accelerated electroconvection.
\newblock {\em Scientific reports}, 6(1):22505, 2016.

\bibitem{fischer2017global}
Andr{\'e} Fischer and J{\"u}rgen Saal.
\newblock Global weak solutions in three space dimensions for electrokinetic flow processes.
\newblock {\em Journal of Evolution Equations}, 17(1):309--333, 2017.

\bibitem{gajewski1986basic}
Herbert Gajewski and Konrad Gr{\"o}ger.
\newblock On the basic equations for carrier transport in semiconductors.
\newblock {\em Journal of mathematical analysis and applications}, 113(1):12--35, 1986.

\bibitem{gao2014high}
Jun Gao, Wei Guo, Dan Feng, Huanting Wang, Dongyuan Zhao, and Lei Jiang.
\newblock High-performance ionic diode membrane for salinity gradient power generation.
\newblock {\em Journal of the American Chemical Society}, 136(35):12265--12272, 2014.

\bibitem{goldman1989electrodiffusion}
David~E Goldman.
\newblock Electrodiffusion in membranes.
\newblock In {\em Membrane Transport: People and Ideas}, pages 251--259. Springer, 1989.

\bibitem{herz2016global1}
Matthias Herz and Peter Knabner.
\newblock Global existence of weak solutions of a model for electrolyte solutions-part 1: Two-component case.
\newblock {\em arXiv preprint arXiv:1605.07396}, 2016.

\bibitem{herz2016global2}
Matthias Herz and Peter Knabner.
\newblock Global existence of weak solutions of a model for electrolyte solutions-part 2: Multicomponent case.
\newblock {\em arXiv preprint arXiv:1605.07445}, 2016.

\bibitem{herz2012existence}
Matthias Herz, Nadja Ray, and Peter Knabner.
\newblock Existence and uniqueness of a global weak solution of a darcy-nernst-planck-poisson system.
\newblock {\em GAMM-Mitteilungen}, 35(2):191--208, 2012.

\bibitem{ignatova2022global}
Mihaela Ignatova and Jingyang Shu.
\newblock Global smooth solutions of the nernst--planck--darcy system.
\newblock {\em Journal of Mathematical Fluid Mechanics}, 24(1):26, 2022.

\bibitem{jasielec2021electrodiffusion}
Jerzy~J Jasielec.
\newblock Electrodiffusion phenomena in neuroscience and the nernst--planck--poisson equations.
\newblock {\em Electrochem}, 2(2):197--215, 2021.

\bibitem{jerome2002analytical}
Joseph~W Jerome.
\newblock Analytical approaches to charge transport in a moving medium.
\newblock {\em Transport Theory and Statistical Physics}, 31(4-6):333--366, 2002.

\bibitem{kato1988commutator}
Tosio Kato and Gustavo Ponce.
\newblock Commutator estimates and the euler and navier-stokes equations.
\newblock {\em Communications on Pure and Applied Mathematics}, 41(7):891--907, 1988.

\bibitem{koch2004biophysics}
Christof Koch.
\newblock {\em Biophysics of computation: information processing in single neurons}.
\newblock Oxford university press, 2004.

\bibitem{lee2016membrane}
Anna Lee, Jeffrey~W Elam, and Seth~B Darling.
\newblock Membrane materials for water purification: design, development, and application.
\newblock {\em Environmental Science: Water Research \& Technology}, 2(1):17--42, 2016.

\bibitem{lee2022global}
Fizay-Noah Lee.
\newblock Global regularity for nernst--planck--navier--stokes systems with mixed boundary conditions.
\newblock {\em Nonlinearity}, 36(1):255, 2022.

\bibitem{lee2018diffusiophoretic}
Hyomin Lee, Junsuk Kim, Jina Yang, Sang~Woo Seo, and Sung~Jae Kim.
\newblock Diffusiophoretic exclusion of colloidal particles for continuous water purification.
\newblock {\em Lab on a Chip}, 18(12):1713--1724, 2018.

\bibitem{lopreore2008computational}
Courtney~L Lopreore, Thomas~M Bartol, Jay~S Coggan, Daniel~X Keller, Gina~E Sosinsky, Mark~H Ellisman, and Terrence~J Sejnowski.
\newblock Computational modeling of three-dimensional electrodiffusion in biological systems: application to the node of ranvier.
\newblock {\em Biophysical journal}, 95(6):2624--2635, 2008.

\bibitem{mock1983analysis}
M.S. Mock.
\newblock {\em Analysis of Mathematical Models of Semiconductor Devices}.
\newblock Advances in numerical computation series. Boole Press, 1983.

\bibitem{mori2009numerical}
Yoichiro Mori and Charles Peskin.
\newblock A numerical method for cellular electrophysiology based on the electrodiffusion equations with internal boundary conditions at membranes.
\newblock {\em Communications in Applied Mathematics and Computational Science}, 4(1):85--134, 2009.

\bibitem{nicholson2000diffusion}
Charles Nicholson, Kevin~C Chen, Sabina Hrab{\v{e}}tov{\'a}, and Lian Tao.
\newblock Diffusion of molecules in brain extracellular space: theory and experiment.
\newblock {\em Progress in brain research}, 125:129--154, 2000.

\bibitem{pods2013electrodiffusion}
Jurgis Pods, Johannes Sch{\"o}nke, and Peter Bastian.
\newblock Electrodiffusion models of neurons and extracellular space using the poisson-nernst-planck equations—numerical simulation of the intra-and extracellular potential for an axon model.
\newblock {\em Biophysical journal}, 105(1):242--254, 2013.

\bibitem{qian1989electro}
Ning Qian and TJ~Sejnowski.
\newblock An electro-diffusion model for computing membrane potentials and ionic concentrations in branching dendrites, spines and axons.
\newblock {\em Biological Cybernetics}, 62(1):1--15, 1989.

\bibitem{rubinstein1990electro}
I.~Rubinstein.
\newblock {\em Electro-diffusion of Ions}.
\newblock Studies in Applied and Numerical Mathematics. Society for Industrial and Applied Mathematics, 1990.

\bibitem{ryham2009existence}
Rolf~J Ryham.
\newblock Existence, uniqueness, regularity and long-term behavior for dissipative systems modeling electrohydrodynamics.
\newblock {\em arXiv preprint arXiv:0910.4973}, 2009.

\bibitem{schmuck2009analysis}
Markus Schmuck.
\newblock Analysis of the navier--stokes--nernst--planck--poisson system.
\newblock {\em Mathematical Models and Methods in Applied Sciences}, 19(06):993--1014, 2009.

\bibitem{yang2019review}
Zi~Yang, Yi~Zhou, Zhiyuan Feng, Xiaobo Rui, Tong Zhang, and Zhien Zhang.
\newblock A review on reverse osmosis and nanofiltration membranes for water purification.
\newblock {\em Polymers}, 11(8):1252, 2019.

\bibitem{zhu2020ion}
Haitao Zhu, Bo~Yang, Congjie Gao, and Yaqin Wu.
\newblock Ion transfer modeling based on nernst--planck theory for saline water desalination during electrodialysis process.
\newblock {\em Asia-Pacific Journal of Chemical Engineering}, 15(2):e2410, 2020.

\end{thebibliography}

\end{document}